\documentclass[12pt,leqno,twoside]{article}
\usepackage[T1]{fontenc}
\usepackage[utf8]{inputenc}
 
\usepackage{amstext,amsmath,amscd, bezier,indentfirst,amsthm,amsgen,enumerate,bm}
\usepackage[left=2.5cm, right=2.5cm, top=2.5cm,bottom=2.5cm]{geometry} 
\usepackage{todonotes}
\usepackage{subcaption}
\usepackage[all,knot,arc,import,poly]{xy}
\usepackage{amsfonts,color}
\usepackage{amssymb}
\usepackage{latexsym}
\usepackage{epsfig}
\usepackage{graphicx}
\usepackage{srcltx}
\usepackage{enumitem}
\usepackage{accents}
\usepackage{imakeidx}
\usepackage{rotating}
\usepackage{verbatim}
\usetikzlibrary{arrows}
\usepackage{arydshln}
\usepackage{csquotes}
\usepackage{tikz-cd}

\makeindex
\usepackage{tikz}
\usetikzlibrary{patterns}
\usetikzlibrary{matrix,arrows,decorations.pathmorphing}
\usepackage{tikz-cd}
\tikzset{commutative diagrams/.cd}
\usepackage{framed,lipsum}
\usepackage{float}
\usepackage{pgfplots}
\usetikzlibrary{patterns}
\usetikzlibrary{intersections}
\usetikzlibrary{positioning,shadows,arrows}
\tikzstyle{every node}=[anchor=west, minimum height=3em]
\usetikzlibrary{decorations.pathmorphing}

\usepackage{wrapfig}
\definecolor{forestgreen}{rgb}{0.00, 0.39, 0.00} 
\definecolor{blueblue}{rgb}{0.40, 0.00, 1.00} 
\definecolor{sienna}{rgb}{0.33, 0.08, 0.11}

\newcommand{\defi}{\textbf}

\theoremstyle{plain}
\newtheorem{theorem}{Theorem}[section] 
\newtheorem*{theorem*}{Theorem}
\newtheorem*{hypothesis*}{Hypothesis}
\newtheorem{lemma}[theorem]{Lemma}
\newtheorem{proposition}[theorem]{Proposition}
\newtheorem{corollary}[theorem]{Corollary}
\newtheorem{notation}[theorem]{Notation}

\theoremstyle{definition}
\newtheorem{definition}[theorem]{Definition}
\newtheorem*{definition*}{Definition}
\newtheorem{remark}[theorem]{Remark}
\newtheorem*{question*}{\sc Question}
\newtheorem*{remark*}{\sc Remark}

\newtheorem{example}[theorem]{Example}
\newtheorem*{example*}{\sc Example}

\DeclareMathOperator{\spn}{span}

\usepackage{mathscinet}
\usepackage[unicode, pdftex, pdfusetitle, plainpages=false,  bookmarks, bookmarksnumbered,colorlinks, linkcolor=blue, citecolor=red,filecolor=black, urlcolor=black]{hyperref}
\usepackage{pgfplots}
\pgfplotsset{width=7cm,compat=1.8}
\usepackage[tight]{minitoc} 
\usepackage{caption}
\makeatletter
\renewcommand*{\numberline}[1]{\hb@xt@1em{#1\hfil}} 
\makeatother

\usepackage[version=4]{mhchem}
\usepackage{chemfig}

\usepackage[backend=bibtex,style=numeric,backref=true,maxbibnames=99]{biblatex}
\usepackage{pgfplots}

\pgfplotsset{width=7cm,compat=1.8}

\usepackage{caption}
\usepackage{titling}

\hypersetup{
    colorlinks=true,
    linkcolor=blue,
    filecolor=magenta,      
    urlcolor=violet,
    pdftitle={Overleaf Example},
    pdfpagemode=FullScreen,
    }

\newcommand\RR{\mathbb{R}}

\newcommand\by{\boldsymbol{y}}
\newcommand\byi{{\boldsymbol{y}_i}}
\newcommand\byj{{\boldsymbol{y}_j}}
\newcommand\byp{{\boldsymbol{y}^\prime}}
\newcommand\be{\boldsymbol{e}}

\newcommand\bk{\boldsymbol{k}}

\newcommand\bx{\boldsymbol{x}}

\newcommand\bv{\boldsymbol{v}}

\newcommand\bK{\boldsymbol{K}}

\newcommand\bX{\boldsymbol{X}}

\newcommand\bp{\mathbf{p}}

\newcommand\bbeta{\boldsymbol\beta}
\newcommand\bB{\mathcal{B}}

\title{The structure of the moduli space of toric dynamical systems of a reaction network}

\date{\today}
\author{Gheorghe Craciun\thanks{University of Wisconsin-Madison, USA}, Jiaxin Jin\thanks{Ohio State University, USA}, Miruna-\c Stefana Sorea\thanks{SISSA (Scuola Internazionale Superiore di Studi Avanzati), Trieste, Italy and  Lucian Blaga University of Sibiu, Romania}}

\thanksmarkseries{arabic}

\begin{document}

\maketitle

\begin{abstract}
\noindent
    We consider \emph{toric dynamical systems}, which are also called \emph{complex-balanced mass-action systems}. These are remarkably stable polynomial dynamical systems that arise from the analysis of mathematical models of reaction networks when, under the assumption of mass-action kinetics, they can give rise to {\em complex-balanced equilibria}. Given a reaction network, we study the moduli space of toric dynamical systems generated by this network, also called  \emph{the toric locus} of the network. The toric locus is an algebraic variety, and we are especially interested in its topological properties. We show that complex-balanced equilibria \emph{depend continuously} on the parameter values in the toric locus, and, using this result, we prove that the toric locus has a remarkable  \emph{product structure}: it is homeomorphic to the product of the set of complex-balanced flux vectors and the affine invariant polyhedron of the network. In particular, it follows that the toric locus is a {\em contractible manifold}.   Finally, we show that the toric locus is {invariant} with respect to bijective affine transformations of the generating reaction network.
\end{abstract}

\renewcommand{\numberline}[1]{#1~}
\setcounter{tocdepth}{2}
\tableofcontents

\section{Introduction}
Nonlinear dynamical systems are among the most common mathematical models used in the study of  population dynamics, epidemiology, biochemistry, just to name a few \cite{cy}. However, the analysis of long-term dynamical behaviour of nonlinear dynamical systems is a very difficult problem. Finding explicit, quantitative answers related to the question of how a system evolves in continuous time is usually impossible. Inspired by the work of Poincaré \cite{Poincare}, mathematicians started tackling the qualitative aspects of these systems. However, this task is also a  difficult one. For instance, consider the second part of Hilbert's 16th problem, concerning polynomial dynamical systems in the real plane. After more than a century, the problem of finding an upper bound for the number of limit cycles remains open even in the quadratic case; for technical details and historical aspects of Hilbert's 16th problem, we refer the reader to \cite{ilyashenko}, \cite[Chapter 11]{MR3235494}. Another example meant to show that nonlinear dynamical systems are challenging is the Lorenz system: a quadratic polynomial dynamical system, in the three-dimensional Euclidean space, which exhibits chaotic dynamics \cite{MR4021434}.  

\subsection{Context}
We focus on polynomial dynamical systems generated by {\em reaction networks}, which are represented by directed
graphs in  Euclidean space. One of the goals of reaction network theory is to determine information about the qualitative long-term dynamics from the algebro-combinatorial structure of the network.  
In order to model the evolution in time of the concentrations of interacting species, we use autonomous systems of ordinary differential equations, dictated by the network structure. Under the assumption of mass-action kinetics (\cite[Section 2.1.2]{feinberg}), this leads to fruitful interactions between the study of  reaction networks and applied algebraic geometry (see \cite{MR4201909}), because these systems have polynomial right-hand side. The law of mass-action is very commonly used in mathematical modeling, for instance in population dynamics, ecology, biochemistry, and chemical engineering \cite{boros2019existence,feinberg,cy}.   

In particular, we are interested in \emph{complex-balanced} mass-action systems (see \cite[Chapter 15]{feinberg}). Introduced by Horn and Jackson in \cite{HJ}, these represent a large class of polynomial dynamical systems that are known   
to have a stable dynamical behaviour that is very desirable in applications. For instance, Horn and Jackson proved that complex-balanced dynamical systems possess exactly one positive equilibrium up to conservation laws (i.e., one within each invariant polyhedron) and that this equilibrium is locally asymptotically stable (see \cite{HJ}, \cite[Theorem 2.3]{cy}). One of the most important lines of research in the field of reaction network theory is \emph{the Global Attractor Conjecture}, which says that this equilibrium is actually \emph{globally} asymptotically stable. This has been already proven in several cases, under various additional hypotheses. For the state of the art, we refer the reader to \cite{cy}. A proof in full generality of the Global Attractor Conjecture has been proposed in~\cite{GlAttrConj}.

Besides their stable dynamical behaviour, another advantage of complex-balanced dynamical systems is the fact that tools from commutative algebra, computational, applied, real algebraic geometry turn out to be useful in deducing qualitative dynamical properties, which are often encoded or hidden in the geometric structure of the associated reaction networks. For instance, in \cite{CDSS} complex-balanced dynamical systems have also been called \emph{toric dynamical systems} by Craciun, Dickenstein, Shiu and Sturmfels, to emphasize their strong combinatorial aspects and the remarkable algebraic properties of their moduli spaces. To be more precise, consider the parameter space of a reaction network. The set of parameters that gives rise to complex-balanced dynamical systems is called \emph{the toric locus} because up to a change of coordinates, this set is a variety given by a binomial ideal, intersected with the positive orthant (see \cite{CDSS}). 
Toric varieties appear in numerous applications \cite{BM} and are very appreciated and well understood by algebraic geometers, who use them often in their quest for examples and counterexamples, due to their combinatorial representation and their computational assets. For a presentation of toric varieties from the point of view of Nonlinear Algebra, the reader may refer to \cite[Chapter 8]{BM};  according to \cite[page 126]{BM}, \enquote{the world is toric}. 
Increasing interest for the moduli spaces of toric dynamical systems has been shown recently. For instance, methods to expand the toric locus from a set of Lebesgue measure zero to a positive measure set using the  \emph{disguised toric locus} are proposed in the form of a systematic algorithm in \cite{bcs}, where the authors leverage some properties of the notion of dynamical equivalence from \cite{CJY20}. See also \cite{haque2022disguised}, where the authors show that the disguised toric locus is invariant under invertible affine transformations of the network.

\subsection{Main contributions}

The results of this paper concern the topological structure of the toric locus. One of our main contributions is to show that the complex-balanced equilibria \emph{depend continuously on the parameter values}, 
i.e., {\em reaction rate constants} (Theorem \ref{th:qCont}). 
We then use this result to prove that the moduli space of toric dynamical systems (i.e., the toric locus) of any reaction network is \emph{connected} (Theorem \ref{th:main}). Next, in Theorem \ref{th:product} we show that the moduli space is \emph{homeomorphic to the product} of the set of complex-balanced flux vectors (Definition \ref{def:fluxVectors}) and the affine invariant polyhedron (Definition \ref{def:stoichiom}). 

Being homeomorphic to the product of two path-connected spaces, it also follows that the toric locus is path-connected. Hence, given any two points in the toric locus of a toric dynamical system, there will exist a continuous path between them. This might be  advantageous in computations, for instance when using numerical methods for constructing the set of equilibria along a path in parameter space. Recall that the main strategy used by homotopy continuation methods is tracking the solutions of  systems of polynomial equations which are easier to solve than the given system, or which are already known (see for instance BERTINI \cite{bert}, Julia HomotopyContinuation \cite{numericalNonlinearAlgebra,juliaHC,MR4178964,MR3984062,MR2160078,MR3840407}). 
Such tracking can take advantage of the path connectivity of the toric locus.

Furthermore,  
we recover a result from \cite[Theorem 9]{CDSS}, which says that the codimension of the toric locus in the parameter space is equal to the deficiency of the network (see Definition \ref{def:defic}). 
We also show that the toric locus is invariant under bijective affine transformations of a network (Theorem~\ref{thm:bijectiveAffine}). This result has recently been extended in~\cite{haque2022disguised}, where the authors show that the {\em disguised} toric locus is also invariant under bijective affine transformations of a network.

\subsection{Structure of the paper}

In Section \ref{sec:preliminaryNotions}, we introduce standard terminology and notations concerning dynamical systems generated by reaction networks, mostly focusing on mass-action complex-balanced dynamical systems, also called toric.  
In Section \ref{sect:dependContin}, we prove that complex-balanced equilibria depend continuously on the parameter values in $\mathcal{V}(G)$. Leveraging this result, in Section \ref{sec:Conn} we show that the toric locus is connected. 
In Section \ref{sec:product}, we first prove that the toric locus is homeomorphic to a product space.
Using this property, in Section \ref{sect:deficiency} we show Proposition \ref{prop:dim} which gives a precise formula for the dimension of the toric locus of the network.
In Section \ref{sect:bijectiveAffineTransf}, we prove Theorem \ref{thm:bijectiveAffine}
showing that any bijective affine transformation of the network preserves the toric locus.

\section{Preliminary notions}\label{sec:preliminaryNotions}

In this section, mostly following \cite{cy}, we present standard terminology  concerning a special class of nonlinear dynamical systems that are generated by (bio-chemical) reaction networks, under the assumption of mass-action kinetics. For an introduction to the general theory of nonlinear dynamical systems, the reader could refer for instance to the textbooks \cite{kuznetsov1998elements, strogatz}.
 
First, we give some classical definitions and notations relevant to the study of mass-action dynamical systems and to (bio-chemical) reaction networks. Next, we present a special class of these systems: \emph{complex-balanced} dynamical systems, which are also called \emph{toric} dynamical systems. More details can be found in the textbooks \cite{feinberg} and \cite{MR4284895}, the latter one with a view toward Nonlinear Algebra. See also \cite{CDSS, bcs, CJY20, MR3920470}. 

\begin{notation} \rm
\begin{enumerate}
\item We let $\mathbb{R}_{\geq 0}^n$ and $\mathbb{R}_{>0}^n$ denote the sets of vectors with non-negative and positive entries respectively. Similarly, $\mathbb{Z}_{\geq 0}^n$ is the set of vectors with non-negative integer components. We denote the cardinality of a set $A$ as $|A|$, and the disjoint union of sets $A$ and $B$ is denoted by $A \sqcup B$.

\item Let us consider two vectors $\bx, \by \in \RR^n$ with $\bx = (x_1, \ldots, x_n)^{\intercal}$ and $\by = (y_1, \ldots, y_n)^{\intercal}$. 
The following are the vector operations that will be used in this paper:
\begin{equation} \notag
\begin{split}
\bx \circ \by & := (x_1 y_1, \ldots, x_n y_n)^{\intercal},
\\ \exp (\bx) & := (\exp(x_1), \ldots, \exp(x_n))^{\intercal},
\\ \ln (\bx) & := (\ln(x_1), \ldots, \ln(x_n))^{\intercal}.
\end{split}
\end{equation}
For $\bx\in \RR^n_{>0}$ we also define $\bx^{\by}  := x_1^{y_{1}} x_2^{y_{2}}  \ldots x_n^{y_{n}}.$

\item We also apply vector operations on a subset of $\RR^n$, where they are applied to all elements of the subset. For example, given a vector $\bx \in \RR^n$ and a set $A \subseteq \RR^n$, 
\begin{equation} \notag
\bx \circ A := \{ \bx \circ \by: \by \in A \}.
\end{equation}
\end{enumerate}
\end{notation}

\subsection{Dynamics of reaction networks with mass-action kinetics}

We work with deterministic, autonomous, and continuous dynamical systems, generated by reaction networks. The goal is to model the variation in time of the concentrations of the species involved, under the assumption of mass-action kinetics. Mostly following the terminology and notations from \cite{cy}, let us give precise  definitions of these classical notions. 

\bigskip

The classical definition of a reaction network involves {\em species, complexes, and reactions}, as illustrated in Figure \ref{fig:3species} (and explained in detail below). In recent work it was observed that the equivalent definition of reaction network as a {\em directed graph embedded in Euclidean space}  leads to very convenient notations, see Figure \ref{fig:3speciesEEgraph}; this is why we employ this definition here. 

\begin{definition}
\begin{enumerate}
\item We denote by  $n$  the \textbf{number of species} involved in the reaction network, and denote by $X_1, \ldots, X_n$  the \textbf{species} of the network.

\item Denote by $x_i$ the \defi{concentration} of the species $X_i$, for $i=1,\ldots, n$. We consider $x_i$ as functions of time $t$: $x_i=x_i(t)$. At any time $t \geq 0$, this gives us a vector $\bx
= (x_1, \ldots, x_n)^\intercal \in \mathbb{R}^n$, also called a \defi{state} of the system.

\item A formal linear combination of species $\{ X_i\}^n_{i=1}$, with non-negative real coefficients is called a \defi{complex}. 
A \defi{reaction}  
is a directed edge between two distinct complexes.
\end{enumerate}
\end{definition}

\begin{definition} \label{def:reactNetw_RateConstants}
A \defi{reaction network}, also called a \defi{Euclidean embedded graph} or \mbox{\defi{E-graph}} (see \cite{MR3920470}) 
is a finite directed graph $G=(V,E)$ such that the set  $V\subset \mathbb{R}^n$ is a finite set of \defi{vertices} and the set $E\subseteq V\times V$ represents the finite set of \defi{edges}. We assume that there are neither self-loops nor isolated vertices.
\begin{enumerate}
\item We denote the \textbf{number of vertices} by $m$, and let $V = \{ \boldsymbol{y}_1, \ldots, \boldsymbol{y}_m\}$, where each vertex $\byi \in V$ corresponds to a complex. The entries of the vertex are the coefficients of the species in the corresponding formal linear combination. 

\item A directed edge connecting two vertices $\byi\in V$ to $\byj\in V$ is denoted by $\byi \rightarrow \byj\in E$ and represents a reaction in the network. We call the difference vector $\byj - \byi\in\mathbb{R}^n$, the \defi{reaction vector}. Here $\byi$ and $\byj$ denote the \defi{source vertex} and \defi{target vertex} respectively. 
\end{enumerate}
\end{definition}

\medskip

As we mentioned above, reaction networks can either be represented as sets of reactions (see Figure~\ref{fig:3species}), or, equivalently, by using the  Euclidean embedded graphs, where the vertices correspond to the complexes (see Figure~\ref{fig:3speciesEEgraph}). This is illustrated in the example below.

\medskip

\begin{example} \label{ex:complexesSpecies} 
Let us consider the reaction network from Figure~\ref{fig:3species}. There are three interacting species: $X_1$, $X_2$, $X_3$, and three complexes:
\begin{equation} \notag
2X_1 + 3X_2, \ 2X_2, \ X_3,
\end{equation}
and four reactions (directed edges between complexes):
\begin{equation} \notag
2X_1+3X_2 \to 2 X_2, \ 2 X_2 \to 2X_1+3X_2, \ 2 X_2 \to X_3, \ X_3 \to 2X_1+3X_2.
\end{equation}
 
\begin{figure}[H] 
\begin{center}
\begin{tikzpicture}
		\node (1) at (-3,2) {$\bullet$};
		\node (3) at (0,0) {$\bullet$};
		\node (2) at (3,2) {$\bullet$};
		\node [left=1pt of 1] {$2X_1+3X_2$};
		\node [right=1pt of 2] {$2X_2$};
		\node (x3) at (0,-0.5) {$X_3$};
		\draw [{->}, -{stealth}, thick, transform canvas={yshift=1.5pt}] (1) -- (2) node [midway, above] {};
		\draw [{->}, -{stealth}, thick, transform canvas={yshift=-1.5pt}] (2) -- (1) node [midway, below] {};
		\draw [{->}, -{stealth}, thick, transform canvas={xshift=1.5pt}] (2) -- (3) node [midway, right=6pt] {};
		\draw [{->}, -{stealth}, thick, transform canvas={xshift=-1.5pt}] (3) -- (1) node [midway, left] {};
\end{tikzpicture}\end{center}
\caption{A reaction network with three species, three complexes, and four reactions.}
\label{fig:3species}
\end{figure}  

The real coefficients appearing in each formal linear combination of species of the complexes of the reaction network from Figure \ref{fig:3species} can be represented by vectors in the three-dimensional Euclidean space:
\begin{equation} \notag
\boldsymbol{y}_1=\begin{pmatrix}2\\ 3\\ 0\end{pmatrix}, \
\boldsymbol{y}_2=\begin{pmatrix}0\\ 2\\ 0\end{pmatrix}, \
\boldsymbol{y}_3=\begin{pmatrix}0\\ 0\\ 1\end{pmatrix}.
\end{equation} 
This gives rise to an E-graph (see Figure \ref{fig:3speciesEEgraph}), whose edges become actual vectors $\byi-\byj\in\mathbb{R}^3$.

\begin{figure}[H]
\centering
 \begin{tikzpicture}[scale=1.3]
    \begin{axis}[
      view={115}{15},
      axis lines=center,
      width=6cm,height=6cm,
      ticks = none, 
xmin=0,xmax=3.5,ymin=0,ymax=3.5,zmin=0,zmax=3,
xlabel=$X_1$,
ylabel=$X_2$,
zlabel=$X_3$,
    ]
    \addplot3 [no marks, dashed, gray!70] coordinates {(1,0,0)  (1,0,2.5)};
    \addplot3 [no marks, dashed, gray!70] coordinates {(2,0,0)  (2,0,2.5)};
    \addplot3 [no marks, dashed, gray!70] coordinates {(0,1,0)  (0,1,2.5)};
    \addplot3 [no marks, dashed, gray!70] coordinates {(0,2,0)  (0,2,2.5)};
    \addplot3 [no marks, dashed, gray!70] coordinates {(2.5,0,1)  (0,0,1) (0,2.5,1)};
    \addplot3 [no marks, dashed, gray!70] coordinates {(2.5,0,2)  (0,0,2) (0,2.5,2)};
    \addplot3 [no marks, dashed, gray!70] coordinates {(1,0,0)  (1,3.5,0) };
    \addplot3 [no marks, dashed, gray!70] coordinates {(2,0,0)  (2,3.5,0) };
    \addplot3 [no marks, dashed, gray!70] coordinates {(0,1,0)  (2.5,1,0) };
    \addplot3 [no marks, dashed, gray!70] coordinates {(0,2,0)  (2.5,2,0) };
    \addplot3 [no marks, dashed, gray!70] coordinates {(0,3,0)  (2.5,3,0) };
    \addplot3 [only marks, black] coordinates {(2,3,0) (0,2,0) (0,0,1)};
    \node [outer sep=1pt] (1) at (axis cs:2,3,0) {};
    \node [outer sep=1pt] (2) at (axis cs:0,2,0) {};
    \node [outer sep=1pt] (3) at (axis cs:0,0,1) {};
        \node at (axis cs:2.6,3.4,0) [left] {$\boldsymbol{y}_1$\,};
        \node at (axis cs:0,2,0.2) [right] {$\boldsymbol{y}_2$};
        \node at (axis cs:0.2,0,1) [left] {$\boldsymbol{y}_3$};
    \draw [-{stealth}, thick, red,transform canvas={xshift=0ex, yshift=3.5ex}] (1)--(2) ;
    \draw [-{stealth}, thick, red,transform canvas={xshift=0.2ex, yshift=-3.5ex}] (2)--(1) ; 
    \draw [-{stealth}, thick, red, transform canvas={xshift=-0.2ex, yshift=0ex}] (2)--(3) ; 
    \draw [-{stealth}, thick, red, transform canvas={xshift=-0.6ex, yshift=-0.2ex}] (3)--(1) ; 
    \end{axis}
    \end{tikzpicture}
    \caption{The Euclidean embedded graph of the network from Figure \ref{fig:3species}.}
    \label{fig:3speciesEEgraph}
\end{figure}

\end{example}

\begin{definition} \label{def:weaklyRev}
Let $G=(V, E)$ be a Euclidean embedded graph.
\begin{enumerate}
\item The set of vertices $V$ is partitioned by its connected components, also called \defi{linkage classes}, and we identify them by the subset of vertices that belong to that connected component. 
We denote the \defi{number of connected components} by $\ell$, and let $V = V_1 \sqcup V_2 \cdots \sqcup V_\ell$, where each $V_i$ represents a connected component of $G$. 

\item A connected component is called  \defi{strongly connected} if every edge is part of an oriented cycle. 
Furthermore, a strongly connected component is said to be \defi{terminal}, if no other strongly connected component is reachable from it.

\item A graph $G=(V,E)$ is \defi{weakly reversible}, if every connected component is strongly connected. 
\end{enumerate}
\end{definition}

We work under the assumption of {\em mass-action kinetics}, which says that the rate with which a reaction takes place is directly proportional to the product of the concentrations of the reactant species (see \cite{cy} and references therein). Under this assumption, the dynamics can be modeled using the ODE system  (\ref{eq:massAction}) below. Starting with the work of Gatermann (see \cite{CDSS, Ga01}), the polynomial structure of the right-hand side of (\ref{eq:massAction}) has given rise to fruitful interactions between the field of reaction networks and the methods of  computational algebra.

\begin{definition} \label{def:massAction}
Given a Euclidean embedded graph $G=(V,E)$, each edge $\byi\rightarrow \byj$ is decorated with a positive constant $k_{\byi \rightarrow \byj}$ or $k_{ij}$, called a \defi{reaction rate constant}. 
Further, we denote  by ${\bk} :=(k_{ij})\in\mathbb{R}_{>0}^{E}$ the \defi{vector of reaction rate constants}.
The \defi{associated mass-action system} generated by $(G,{\bk})$ 
on $\RR^n_{>0}$ is given by
\begin{equation} \label{eq:massAction}
\frac{\mathrm{d}\bx}{\mathrm{d} t}= \sum_{\byi \rightarrow \byj \in E}k_{\byi \rightarrow \byj} \bx^\byi(\byj -\byi).
\end{equation}
\end{definition}

\medskip

For example, consider the Euclidean embedded graph from Figure~\ref{fig:3speciesEEgraph} (see Example \ref{ex:complexesSpecies}). Under  mass-action kinetics, 
the associated dynamical system is
\begin{equation} \label{eq1}
\begin{split}
\frac{\mathrm{d}\bx}{\mathrm{d} t}&=k_{12} x_1^2 x_2^3  ({\boldsymbol{y}_2}-{\boldsymbol{y}_1})+k_{21} x_2^2 ({\boldsymbol{y}_1}-{\boldsymbol{y}_2}) + k_{23} x_2^2 ({\boldsymbol{y}_3}-{\boldsymbol{y}_2})+ k_{31} x_3 ({\boldsymbol{y}_1}-{\boldsymbol{y}_3})\\
&=k_{12} x_1^2 x_2^3  \begin{pmatrix} -2 \\ -1\\ 0 \end{pmatrix}+k_{21} x_2^2 \begin{pmatrix} 2 \\ 1\\ 0 \end{pmatrix} + k_{23} x_2^2 \begin{pmatrix} 0 \\-2\\1 \end{pmatrix}+ k_{31} x_3 \begin{pmatrix} 2 \\3\\-1 \end{pmatrix}\\
    &=\begin{pmatrix} -2 k_{12} x_1^2 x_2^3 + 2k_{21} x_2^2+2 k_{31}x_3 \\ -k_{12} x_1^2 x_2^3+(k_{21}-2k_{23})x_2^2+3k_{31} x_3\\ k_{23}x_2^2-k_{31}x_3\end{pmatrix}.
\end{split}
\end{equation}

\medskip

Before the end of this subsection, we define  {\em affine invariant polyhedrons}; they will play an important role in the proof of our main results, starting with Section \ref{sect:dependContin}.

\begin{remark}[\cite{cy}]
\label{rmk: invariant}
Note that we set the domain of \eqref{eq:massAction} to be $\mathbb{R}_{>0}^n$. 
In general, systems of ODEs do not allow
$\mathbb{R}_{>0}^n$ to be forward-invariant. But under the assumption that $V \subset \mathbb{Z}_{\geq 0}^n$, the positive orthant $\mathbb{R}_{>0}^n$ is forward-invariant under system \eqref{eq:massAction}. See also \cite[Remark 2.3, page 3]{haque2022disguised}: we could also allow $V \subset \mathbb{R}_{\geq 0}^n$ or $V \subset \mathbb{R}^n.$
\end{remark}

\begin{definition} \label{def:stoichiom}
Let $G = (V, E)$ be a Euclidean embedded graph. We denote the \defi{stoichiometric subspace} of $G$ by $\mathcal{S}$, which is
\begin{equation}
\mathcal{S} = \spn \{ \byj - \byi: \byi \rightarrow \byj \in E \}.
\end{equation}
By Remark \ref{rmk: invariant}, any solution to \eqref{eq:massAction}
with initial condition ${\boldsymbol{x}_0}\in \mathbb{R}_{>0}^n$ and $V \subset \mathbb{Z}_{\geq 0}^n$, is confined to $({\boldsymbol{x}_0} + \mathcal{S})\cap \mathbb{R}_{>0}^n$. The set  $({\boldsymbol{x}_0} + \mathcal{S})\cap \mathbb{R}_{>0}^n$ is called the \defi{affine invariant polyhedron} of ${\boldsymbol{x}_0}$.
For the sake of simplicity, we use the following notation:
\[
\mathcal{S}_{{\boldsymbol{x}_0}} := ({\boldsymbol{x}_0} + \mathcal{S})\cap \mathbb{R}_{>0}^n.
\]
\end{definition}

\subsection{Complex-balanced dynamical systems and their properties}

The importance of complex-balanced dynamical systems is mostly due to their strong stability properties. For more details, we advise the reader to consult \cite{HJ},  \cite[Theorem 2.3]{cy}. Using a strictly convex Lyapunov function, Horn and Jackson  proved in \cite{HJ} that if a mass-action system has a complex-balanced steady state, then all its positive steady states are also complex-balanced, and that  there is a unique and locally asymptotically stable steady state within each affine invariant polyhedron.

\begin{definition} \label{def:CB}
Consider the associated mass-action system generated by 
$(G, \bk)$:
\begin{equation} \notag
\frac{\mathrm{d}\bx}{\mathrm{d} t}= \sum_{\byi \rightarrow \byj \in E}k_{\byi \rightarrow \byj} \bx^\byi(\byj -\byi).
\end{equation}
A state ${\boldsymbol{x}^*} \in \mathbb{R}_{>0}^n$ is called a \defi{positive steady state} if 
\begin{equation} \label{eq:ss}
\sum_{\byi \rightarrow \byj \in E}k_{\byi \rightarrow \byj} ({\boldsymbol{x}^*})^\byi(\byj -\byi) = \mathbf{0}.
\end{equation}
A positive steady state ${\boldsymbol{x}^*} \in \mathbb{R}_{>0}^n$ is called a \defi{complex-balanced steady state} if at each vertex $\by_0 \in V$,
\begin{equation} \label{eq:cB}
\sum_{\by_0 \to \byp \in E} k_{\by_0 \to \byp} ({\boldsymbol{x}^*})^{\by_0}
= \sum_{\by \to \by_0 \in E}k_{\by \to \by_0}
({\boldsymbol{x}^*})^{\by}.
\end{equation}
We say the pair $(G,{\bk})$ \defi{satisfies the complex-balanced conditions} if it has a complex-balanced steady state; and the mass-action system generated by $(G, \bk)$ is called \defi{a complex-balanced system} or \defi{toric dynamical system}. 
\end{definition}

The following classical theorem illustrates some of the most important dynamical properties of complex-balanced systems.   

\begin{theorem}
[{\cite[Theorem 2.3]{cy}}]
\label{thm:HJ}
Consider a complex-balanced system $(G, \bk)$ with one complex-balanced steady state $\bx^* \in \RR^n_{>0}$. Denote its  associated stoichiometric subspace by $\mathcal{S}$.  Then the following hold:
\begin{enumerate}[label=(\alph*)]
\item All positive steady states are complex-balanced. There is exactly one steady state within each invariant polyhedron. 

\item Any complex-balanced steady state $\bx$ satisfies the following relation: $\ln \bx - \ln \bx^* \in \mathcal{S}^\perp$.  

\item Every complex-balanced steady state is locally asymptotically stable within its invariant polyhedron. 
\end{enumerate}
\end{theorem}

\medskip

Moreover, the mass-action system (\ref{eq:massAction}) admits a matrix decomposition, which helps us in studying complex-balanced steady states.
Recall that the number of species is denoted by $n$, and the number of vertices is denoted by $m$.
Following \cite{CDSS}, we set the $n\times m$ matrix $Y$, whose columns correspond to vertices:
\[
Y := ({\boldsymbol{y}_1}, {\boldsymbol{y}_2}, \ldots, {\boldsymbol{y}_m}) = (y_{ji}) \in \mathbb{R}^{n \times m},
\]
Next, we build the following vector of monomials: 
\begin{align} \notag
\Psi(\bx) :=
\begin{pmatrix}
           \bx^{\boldsymbol{y}_1} \\
           \vdots \\
           \bx^{\boldsymbol{y}_m}
\end{pmatrix} \in \mathbb{R}^m.
\end{align}
Since each directed edge $\byi \rightarrow \byj \in E$ has a reaction rate constant $k_{ij}\in\mathbb{R}_{>0}$, we construct the $m \times m$ \textbf{Kirchoff matrix} $A_{\bk}$, which is the transpose of the negative of the graph Laplacian of $(V, E, {\bk})$:
\begin{equation} \label{def: Ak}
[A_{\bk}]_{ji} :=
\begin{cases}
k_{\byi \rightarrow \byj}, & \ \text{if } \   i \neq j \ \text{and } \byi \rightarrow \byj \in E \\[5pt]
-\sum\limits_{\byi \rightarrow \byj \in E} k_{\byi \rightarrow \byj}, & \ \text{if } \ i =j, \\[5pt]
0, & \ \text{otherwise}.
\end{cases}
\end{equation}

Then the mass-action dynamical system (\ref{eq:massAction}) generated by $(G, \bk)$ can be written in the following vectorial representation:
\begin{align} \label{eq:psi}
\frac{\mathrm{d}\bx}{\mathrm{d} t} 
= Y\cdot A_{\bk}\cdot \Psi(\bx).
\end{align}

\begin{remark}
Note that the notation we use in (\ref{eq:psi}) is different from the one in \cite[page 2]{CDSS}, by transposing. 
\end{remark}

Under direct computation, the $i$-th component of
$A_{\bk} \cdot \Psi(\bx)$ is
\begin{equation} \notag
[A_{\bk} \cdot \Psi(\bx)]_i
= \sum_{\byj \to \byi \in E}k_{\byj \to \byi} \bx^\byj
-   \sum_{\byi \to \byj \in E} k_{\byi\to \byj} \bx^\byi.
\end{equation}
Therefore, the equality (\ref{eq:cB}) is equivalent to 
\begin{equation} \label{eq:psiEquiv}
A_{\bk}\cdot \Psi(\boldsymbol{x}^*)=\mathbf{0}
\end{equation}
where ${\boldsymbol{x}^*} \in \mathbb{R}_{>0}^n$ is a complex-balanced steady state for the mass-action system $(G, {\bk})$.

\medskip

The following Lemma \ref{thm:Ak kernel} is a key 
result that we will use in the proof of Proposition \ref{prop:iff}, where we give a characterization of the complex-balanced equilibria. 

\begin{lemma}[{\cite[page 94]{MR413930}}]
\label{thm:Ak kernel}
Consider a mass-action system $(G,\bk)$ with terminal strongly connected components $T_1, T_2, \ldots, T_t$ and vertices $\{\by_1, \by_2, \ldots, \by_m \}$. Then $\text{ker}(A_{\bk})$ (see equation (\ref{def: Ak}) for the definition of $A_{\bk}$) has a basis $\{\be_1, \ldots, \be_t\}$, such that
\begin{equation} \notag
\be_p = 
\begin{cases}
[\be_p]_i > 0, & \ \text{if } \byi \in T_p, \\[5pt] 
[\be_p]_i = 0, & \ \text{otherwise},
\end{cases}
\end{equation}
where $1 \leq i \leq m$ and $1 \leq p \leq t$.
\end{lemma}

For the proof of Lemma \ref{thm:Ak kernel}, we refer the reader to \cite[page 94]{MR413930} or  \cite[Theorem 4.2]{Gu03}.

\begin{example}
Revisiting Example \ref{ex:complexesSpecies}, we have 
\begin{align*}
Y = ({\boldsymbol{y}_1}, {\boldsymbol{y}_2}, {\boldsymbol{y}_3}) =
\begin{pmatrix}
2  & 0 & 0\\
3 & 2 & 0\\
0 & 0 & 1
\end{pmatrix},
\end{align*}
and
\begin{align*}
A_{\bk}=\begin{pmatrix}
-k_{12} & k_{21} & k_{31}\\
k_{12} & -k_{21}-k_{23} & 0\\
0 & k_{23} & -k_{31}
\end{pmatrix}, \ \
\Psi(\bx) =
\begin{pmatrix}
\bx^{\boldsymbol{y}_1} \\
\bx^{\boldsymbol{y}_2} \\
\bx^{\boldsymbol{y}_3}
\end{pmatrix} = 
\begin{pmatrix}x_1^2 x_2^3
\\ x_2^2
\\ x_3
\end{pmatrix}.
\end{align*} 
Following Equations (\ref{eq:psi}), we derive 
\begin{equation} \notag
\begin{split}
\frac{\mathrm{d}\bx}{\mathrm{d} t} =Y\cdot A_{\bk}\cdot \Psi(\bx)
=\begin{pmatrix} -2 k_{12} x_1^2 x_2^3 + 2k_{21} x_2^2+2 k_{31}x_3 \\ -k_{12} x_1^2 x_2^3+(k_{21}-2k_{23})x_2^2+3k_{31} x_3\\ k_{23}x_2^2-k_{31}x_3\end{pmatrix},
\end{split}
\end{equation}
which gives the same ODE system as (\ref{eq1}).
\end{example}

\subsection{The toric locus \texorpdfstring{$\mathcal{V}(G)$}{Vg}}

Here we introduce the notion of \emph{toric locus}, which is a key concept in this paper. 
See also \cite[Definition 2.2]{bcs}.

\begin{definition} \label{def:VkG}
Consider a Euclidean embedded graph $G=(V,E)$, we let $\mathcal{V}(G) \subseteq \mathbb{R}_{>0}^{E}$ denote the set of parameters ${\bk}\in\mathbb{R}_{>0}^{E}$, for which the dynamical system generated by $(G,{\bk})$ is toric (i.e., complex-balanced). 
We refer to $\mathcal{V}(G)$ as the \defi{moduli space} or the  \defi{toric locus} of toric dynamical systems given by the Euclidean embedded graph $G$.
\end{definition}

The following theorem 
shows us that only weakly reversible E-graphs can give rise to complex-balanced mass action systems.

\begin{theorem}[\cite{horn1972necessary}]
\label{thm:cb connection with graph}

Every Euclidean embedded graph which generates a complex-balanced mass action system is weakly reversible.
Moreover, every Euclidean embedded graph which is weakly reversible permits complex-balanced mass action systems.
\end{theorem}

As a consequence, given an E-graph $G=(V,E)$, we conclude that
\begin{itemize}
\item If $G=(V,E)$ is weakly reversible, then $\mathcal{V}(G) \neq \emptyset$.

\item If $G=(V,E)$ is not weakly reversible, then $\mathcal{V}(G) = \emptyset$.
\end{itemize}
Since we are not interested in the case when $\mathcal{V}(G)$ is empty, we always assume the Euclidean embedded graph $G=(V, E)$ is weakly reversible when working with $\mathcal{V}(G)$ in this paper.

\medskip

In practice, it is difficult to compute precise values for the parameters $k_{ij}\in\mathbb{R}_{>0}$, so we usually  choose a symbolic approach and consider them as unspecified parameters, as in \cite{CDSS}. 
For instance, in Example \ref{ex:complexesSpecies}, suppose $\bx = (x_1, x_2, x_3)$ is a complex-balanced steady state, then the complex-balanced conditions are as follows:
\begin{align*}
k_{21}x_2^2+k_{31}x_3&=k_{12}x_1^2 x_2^3 , \\
k_{12}x_1^2 x_2^3&=k_{21}x_2^2 + k_{23} x_2^2, \\
k_{23}x_2^2&=k_{31} x_3.
\end{align*}
Surprisingly, the toric locus $\mathcal{V}(G)$ in Example \ref{ex:complexesSpecies} is the whole positive orthant $\mathbb{R}_{>0}^4$. This follows from a classical result, known as the \emph{Deficiency Zero Theorem}. We will revisit this example and show the details in Section \ref{sect:deficiency}.

For small enough Euclidean embedded graphs, one can successfully use Computer Algebra software such as Macaulay2 \cite{M2}, in order to apply Elimination theory \cite[Chapter 4]{BM} or Real quantifier elimination \cite[Chapter 12.3]{BPR} for computing the toric locus $\mathcal{V}(G)$.

In general, the toric locus can have quite a complicated algebraic description and it is   not easy to study.
This is reflected by Example \ref{ex:codim2} below, which shows that even for simple Euclidean embedded graphs, the topological structure of the moduli spaces of toric dynamical systems can be interesting. 

\begin{example} \label{ex:codim2}
Consider the mass-action system $(G,{\bk})$ in Figure \ref{fig:4nodes}, with four vertices:
\begin{equation} \notag
{\boldsymbol{y}_1}=\begin{pmatrix} 3 \\ 0 \end{pmatrix}, \ \
{\boldsymbol{y}_2}=\begin{pmatrix} 2 \\ 1 \end{pmatrix}, \ \
{\boldsymbol{y}_3}=\begin{pmatrix} 1 \\ 2 \end{pmatrix}, \ \
\boldsymbol{y}_4=\begin{pmatrix} 0 \\ 3 \end{pmatrix}.
\end{equation} 

Suppose $\bx = (x_1, x_2)$ is a complex-balanced steady state, then the complex-balanced conditions follow:
\begin{align*}
	k_{14} x_1^3 = k_{21} x_1^2x_2 = k_{32} x_1x_2^2 = k_{43} x_2^3.
\end{align*}
By eliminating $x_1, x_2$ above, the moduli space $\mathcal{V}(G)\subset \mathbb{R}_{>0}^4$ is the following algebraic variety given by equations (\ref{eq1_semialg}) and (\ref{eq2_semialg}), intersected with the positive orthant.
\begin{equation} \label{eq1_semialg}
(k_{43}k_{32}k_{21})(k_{21}k_{14}k_{43})=(k_{14}k_{43}k_{32})^2,
\end{equation}
and 
\begin{equation} \label{eq2_semialg}
(k_{14}k_{43}k_{32})(k_{32}k_{21}k_{14})=(k_{21}k_{14}k_{43})^2.
\end{equation}

\begin{figure}[!ht]
\centering
\begin{tikzpicture}[scale=0.9]
      \draw[ultra thin,lightgray] (-0.1,-0.1) grid (3.6,3.6);
      \draw[->,semithick] (-0.2,0) -- (4,0) node[right] {$X_1$};
      \draw[->,semithick] (0,-0.2) -- (0,4) node[above] {$X_2$};
      \node[label=right:{$\boldsymbol{y}_1$}] at (2.55,-.35) {};
      \node[label=right:{}](A) at (2.85,0) {};
      \node[label=right:{$\boldsymbol{y}_2$}](B) at (1.85,1) {};
      \node[label=right:{$\boldsymbol{y}_3$}](C) at (0.85,2) {};
      \node[label=right:{$\boldsymbol{y}_4$}](D) at (-0.15,3) {};
      \draw[fill] (A) circle (2pt);
      \draw[fill] (B) circle (2pt);
      \draw[fill] (C) circle (2pt);
      \draw[fill] (D) circle (2pt);
      \def\shf{0.22ex};
      \draw[->,transform canvas={xshift=\shf,yshift=\shf},thick, red] (B) -- (A);
      \draw[->,transform canvas={xshift=-\shf,yshift=-\shf},thick, red] (A) -- (D);
      \draw[->,transform canvas={xshift=\shf,yshift=\shf},thick,red] (C) -- (B);
      \draw[->,transform canvas={xshift=\shf,yshift=\shf},thick,red] (D) -- (C);
      \end{tikzpicture}
\caption{Cycle on four vertices.}
\label{fig:4nodes}
\end{figure}

After a change of variables,  
the moduli space becomes the intersection of a toric variety with the positive orthant. More precisely, equations (\ref{eq1_semialg}) and (\ref{eq2_semialg}) become 
\begin{equation} 
K_1K_3-K_2^2=0,
\end{equation}
and
\begin{equation} 
K_2K_4-K_3^2 = 0,
\end{equation}
where we set $K_1:=k_{43}k_{32}k_{21}, \ K_2:=k_{14}k_{43}k_{32}, \ K_3:=k_{21}k_{14}k_{43}, \ K_4:=k_{32}k_{21}k_{14}$.

From the algebraic point of view, binomial equations are desirable in computations and toric varieties are a cornerstone of algebraic geometry, since they provide many tractable examples due to their combinatorial structure, which is well understood \cite[Chapter 8]{BM}.
\end{example}

\section{Complex-balanced equilibria depend continuously on the parameter values in the toric locus~\texorpdfstring{$\mathcal{V}(G)$}{VG}}
\label{sect:dependContin}

In this section, we show the first main result of this paper: complex-balanced equilibria depend continuously on the parameters ${\bk}$ in the toric locus $\mathcal{V}(G)$ (see Definition \ref{def:VkG}). 

Now we introduce a map from $\mathcal{V}(G)$ to 
$\mathcal{S}_{{\boldsymbol{x}_0}}$, which is crucial in the later proofs.

\begin{definition} \label{def:q}
Let $G = (V, E)$ be a weakly reversible E-graph with the stoichiometric subspace $\mathcal{S}$.
Given a state ${\boldsymbol{x}_0} \in \mathbb{R}_{>0}^n$, we define the following map:
\begin{equation} \label{equ:q}
Q_{{\boldsymbol{x}_0}}: \mathcal{V}(G) \rightarrow ({\boldsymbol{x}_0}+\mathcal{S}) \cap \mathbb{R}_{>0}^n,
\end{equation}
such that for any $\bk \in \mathcal{V}(G)$, $Q_{{\boldsymbol{x}_0}}({\bk})$ 
is the complex-balanced equilibrium in the invariant polyhedron $\mathcal{S}_{{\boldsymbol{x}_0}}$, under the mass-action system $(G, \bk)$.
\end{definition}

The map $Q_{{\boldsymbol{x}_0}}$ is well-defined for any state ${\boldsymbol{x}_0} \in \mathbb{R}_{>0}^n$ and $\bk \in \mathcal{V}(G)$. This follows  from Theorem \ref{thm:HJ}, where every complex-balanced system admits a unique equilibrium within each invariant polyhedron.  
Now we show some basic properties of the map $Q_{{\boldsymbol{x}_0}}$.

\begin{lemma} \label{lemma:Qsurj}
For any state ${\boldsymbol{x}_0} \in \mathbb{R}_{>0}^n$,
the map $Q_{{\boldsymbol{x}_0}}$ from Definition \ref{def:q}
is surjective.
\end{lemma}

\begin{proof}
To prove the surjectivity of $Q_{{\boldsymbol{x}_0}}$, we show that for any point ${\hat{\boldsymbol{x}}}\in ({\boldsymbol{x}_0}+\mathcal{S})\cap \mathbb{R}_{>0}^n$, there  exists $\hat{\bk} \in \mathcal{V}(G)$ such that $Q_{{\boldsymbol{x}_0}}(\hat{\bk})={\hat{\boldsymbol{x}}}$.

From Definition \ref{def:VkG}, given some parameters ${\bk} \in \mathcal{V}(G)$, there exists $\bx \in \mathbb{R}_{>0}^n$ such that $Q_{{\boldsymbol{x}_0}} ({\bk})={\bx}$ and the pair $(\bk, \bx)$ satisfies the complex-balanced conditions (\ref{eq:cB}), namely: for each vertex $\byi \in V$,
\begin{equation} \label{eq: cb in surjective 1}
\sum_{\byi\to \byj \in E}k_{\byi\to \byj} {\bx}^\byi
= \sum_{\byj\to \byi\in E}k_{\byj\to \byi}
{\bx}^{\byj}.
\end{equation}

Now we define the set of parameters 
$\hat{\bk} = (\hat{k}_{\byi \rightarrow \byj})$ as
\begin{equation} \label{eq:fluxNew}
\hat{k}_{\byi \rightarrow \byj}:=\frac{k_{\byi \rightarrow \byj} {\bx}^\byi} 
{{\hat{\boldsymbol{x}}}^\byi}.
\end{equation}
From \eqref{eq: cb in surjective 1} and \eqref{eq:fluxNew}, we derive that for each vertex $\byi \in V$,
\begin{equation}\label{eq:cbForNewk}
\sum_{\byi\to \byj\in E} \hat{k}_{\byi\to \byj} {\hat{\boldsymbol{x}}}^\byi
= \sum_{\byi\to \byj\in E}k_{\byi\to \byj} {\bx}^\byi
= \sum_{\byj\to \byi\in E}k_{\byj\to \byi}
{\bx}^{\byj}
= \sum_{\byj\to \byi\in E}\hat{k}_{\byj\to \byi}
{\hat{\boldsymbol{x}}}^{\byj}.
\end{equation}
It is clear that $\hat{\bk}\in\mathbb{R}_{>0}^{E}$. Thus, from \eqref{eq:cbForNewk} we get $\hat{\bk} \in \mathcal{V}(G)$ and the pair $(\hat{\bk}, {\hat{\boldsymbol{x}}})$ satisfies the complex-balanced conditions (\ref{eq:cB}).
Hence, we conclude $Q_{{\boldsymbol{x}_0}}(\hat{\bk})={\hat{\boldsymbol{x}}}$.
\end{proof}

\begin{lemma} \label{lemma:ConnFibre}
For any state ${\boldsymbol{x}_0} \in \mathbb{R}_{>0}^n$, consider the map $Q_{{\boldsymbol{x}_0}}$ from Definition \ref{def:q}.
Given any state $\bx \in ({\boldsymbol{x}_0}+\mathcal{S})\cap \mathbb{R}_{>0}^n$, the preimage $Q_{{\bx_0}}^{-1}(\bx)$ is connected.
\end{lemma}

\begin{proof}
Suppose any $\bx \in ({\boldsymbol{x}_0}+\mathcal{S})\cap \mathbb{R}_{>0}^n$. From Lemma \ref{lemma:Qsurj}, we have $Q_{{\bx_0}}^{-1}(\bx) \neq \emptyset$.
Follow Definition \ref{def:q},  for any $\bk \in Q_{{\bx_0}}^{-1}(\bx) \subset \mathcal{V}(G)$, the pair $(\bk, \bx)$ satisfies the complex-balanced conditions, such that for each vertex $\byi \in V$,
\begin{equation} \label{eq:massAct}
\sum_{\byi\rightarrow \byj \in E} k_{\byi \rightarrow \byj}{\boldsymbol{x}}^\byi
= \sum_{\byj \rightarrow \byi \in E} k_{\byj \rightarrow \byi}{\bx}^\byj.
\end{equation} 

Now we claim that the fiber $Q_{{\bx_0}}^{-1}(\bx)$ is a convex set.
Suppose both $\bk^{*}, \bk^{**} \in \mathcal{V}(G)$ satisfy  \eqref{eq:massAct}. 
We will show that any convex combination of $\bk^{*}$ and $ \bk^{**}$ also satisfies \eqref{eq:massAct}.
Let us  consider the following set:
\begin{equation}
L (\bk^{*}, \bk^{**})
:= \{ a \bk^{*} + (1-a) \bk^{**}: 0 \leq a \leq 1\}.
\end{equation}
Under direct computation, we obtain for each vertex $\byi \in V$ and any $0 \leq a \leq 1$,
\begin{equation}
\sum_{\byi\rightarrow \byj \in E}
(a k^{*}_{\byi \rightarrow \byj} + (1-a) k^{**}_{\byi \rightarrow \byj}) {\bx}^\byi
= \sum_{\byj\rightarrow \byi \in E}(a k^{*}_{\byj \rightarrow \byi} + (1-a) k^{**}_{\byj \rightarrow \byi}) {\bx}^\byj.
\end{equation}
Hence, we prove that $L (\bk^{*}, \bk^{**}) \subseteq Q_{{\bx_0}}^{-1}(\bx)$.  
This shows the preimage $Q_{{\bx_0}}^{-1}(\bx)$ is a convex set, and we conclude $Q_{{\bx_0}}^{-1}(\bx)$ is connected.
\end{proof}

\begin{lemma}\label{lemma:Qopen}
For any state ${\boldsymbol{x}_0} \in \mathbb{R}_{>0}^n$,
the map $Q_{{\boldsymbol{x}_0}}$ from Definition \ref{def:q}
is open.
\end{lemma}

\begin{proof}
Pick a point ${\bk} \in \mathcal{V}(G)$, we consider an open neighborhood $U$ of ${\bk}$, such that
\begin{equation} \notag
{\bk}\in U \subseteq \mathcal{V}(G).
\end{equation}
Assume $Q_{{\boldsymbol{x}_0}}({\bk}) = {\bx}$, it suffices for us to prove that ${\bx}$ is in the interior of $Q_{{\boldsymbol{x}_0}}(U)$.
Hence, it is equivalent to show that for any $0 < \epsilon \ll 1$, there exists $\delta > 0$ such that for all ${\hat{\boldsymbol{x}}}$ satisfying
$\| {\hat{\boldsymbol{x}}} - {\bx} \| \leq \delta$, there is a point $\hat{\bk} \in \mathcal{V}(G)$, such that ${\hat{\boldsymbol{x}}}=Q_{{\boldsymbol{x}_0}}(\hat{\bk})$ and $\| \hat{\bk} - \bk \| \leq \epsilon$.

Here we define the set of parameters 
$\hat{\bk} = (\hat{k}_{\byi \rightarrow \byj})$ as
\begin{equation} \label{eq:contin}
\hat{k}_{\byi \rightarrow \byj}:=\frac{k_{\byi \rightarrow \byj} {\bx}^\byi} 
{{\hat{\boldsymbol{x}}}^\byi}.
\end{equation}
Using Lemma \ref{lemma:Qsurj}, we get
$\hat{\bk} \in \mathcal{V}(G)$ and $Q_{{\boldsymbol{x}_0}} (\hat{\bk})={\hat{\boldsymbol{x}}}$.
Moreover, we rewrite \eqref{eq:contin} as 
\begin{equation} \notag
\frac{\hat{k}_{\byi \rightarrow \byj}}{{k}_{\byi \rightarrow \byj}}
= \frac{{\bx}^\byi }{{\hat{\boldsymbol{x}}}^\byi}.
\end{equation}
For each reaction $\byi \rightarrow \byj \in E$ and $0 < \epsilon \ll 1$, 
the continuity of the function $\boldsymbol{x}^\byi$ guarantees the existence of 
$\delta_{\byi \rightarrow \byj}$, such that for any $\| {\hat{\boldsymbol{x}}} - {\bx} \| \leq \delta_{\byi \rightarrow \byj}$,
\[
| \hat{k}_{\byi \rightarrow \byj} - k_{\byi \rightarrow \byj} | \leq \epsilon / |E|.
\]
Then we work on all reactions in $E$ and set $\delta = \min\limits_{\byi \rightarrow \byj \in E} \{ \delta_{\byi \rightarrow \byj} \}$.
Now suppose any state ${\hat{\boldsymbol{x}}}$ satisfying $\| {\hat{\boldsymbol{x}}} - {\bx} \| \leq \delta$, we derive that
\begin{equation} \notag
\| \hat{\bk} - \bk \|
\leq \sum\limits_{\byi \rightarrow \byj \in E} | \hat{k}_{\byi \rightarrow \byj} - k_{\byi \rightarrow \byj} |
\leq \epsilon.
\end{equation}
Therefore, ${\bx}$ is in the interior of $Q_{{\boldsymbol{x}_0}}(U)$ and the proof is concluded. 
\end{proof}

Now we state the main result of this section, Theorem \ref{th:qCont}. We will use this result in the following sections, for proof of the connectedness of the toric locus.

\begin{theorem} \label{th:qCont}
For any state ${\boldsymbol{x}_0} \in \mathbb{R}_{>0}^n$,
the map $Q_{{\boldsymbol{x}_0}}$ from Definition \ref{def:q} is continuous.
In other words, the complex-balanced equilibrium within the invariant polyhedron $S_{\boldsymbol{x}_0}$ depends continuously on the parameter values in $\mathcal{V}(G)$.
\end{theorem}

Theorem \ref{th:qCont} represents a crucial step in the proof of Theorem  \ref{th:product} (more precisely, in Lemma \ref{prop:contInv}), where we will show the product structure of the toric locus. Before proving Theorem \ref{th:qCont}, we need to address some necessary notations and lemmas.

\begin{definition}
Let $G = (V, E)$ be a strongly connected E-graph.
\begin{enumerate}
\item We call $\mathcal{T}$ a \defi{spanning tree} of $G$, if it is a connected, acyclic subgraph of $G$ that contains all vertices in $V$.

\item For a spanning tree $\mathcal{T}$ of $G$, the vertex $\boldsymbol{y} \in V$ is called a \defi{sink} of $\mathcal{T}$, if $\boldsymbol{y}$ is the target vertex for all reactions in $\mathcal{T}$ involving $\boldsymbol{y}$.

\item For a spanning tree $\mathcal{T}$ of $G$ and a vertex $\byi \in V$, 
then we call $\mathcal{T}$ a \defi{spanning $\by_i$-tree} (or \defi{$i$-tree}) if $\byi$
is the only sink of $\mathcal{T}$.
\end{enumerate}
\end{definition}

\begin{notation} \label{not:K} \rm
Let $G = (V, E)$ be a strongly connected E-graph.
\begin{enumerate}
\item Consider a spanning tree $\mathcal{T}$ of $G$, we denote by ${\bk}^\mathcal{T}$  the product of all the reaction rate constants associated with reactions in the spanning tree $\mathcal{T}$.

\item Consider every spanning $\by_i$-tree of $G$, let $K_i$ denote the sum of all products associated with spanning $\by_i$-trees, such that
\begin{equation} \notag
K_i := \sum_{\mathcal{T} \text{an } i \text{-tree}}{\bk}^\mathcal{T}.
\end{equation}
\end{enumerate}
\end{notation}

\begin{proposition}[{\cite[Proposition 3]{CDSS}}]\label{prop:minors}
Consider a mass-action system $(G,{\bk})$ with the strongly connected E-graph $G = (V, E)$. 
Let $A_{{\bk}}$ be its corresponding Kirchoff matrix $A_{{\bk}}$ (see equation (\ref{def: Ak}) for the definition of $A_{\bk}$), and $\mathcal{M}_{i}$ be the matrix obtained by removing the $i$-th row and the $i$-th column of $A_{{\bk}}$, then
\begin{equation}
\mathrm{det}(\mathcal{M}_{i})=(-1)^{m-1} K_i,
\end{equation}
where 
$K_i = \sum\limits_{\mathcal{T} \text{an } i \text{-tree}}{\bk}^\mathcal{T}$ defined in Notation \ref{not:K}.
\end{proposition}

The following Proposition \ref{prop:iff} gives a characterization of the complex-balanced equilibria. The similar conclusion can be obtained from \cite{CDSS}.
For the completeness of the paper, we sketch
the proof here.

\begin{proposition}
\label{prop:iff}
Consider a weakly reversible mass-action system $(G,{\bk})$ with $\ell$ connected components. 
For any two vertices $\byi$ and $\byj$, we construct the following equation:
\begin{equation}\label{eq:binom}
    K_i \bx^{\byj}-K_j \bx^{\byi}=0,
\end{equation}
where 
$K_i = \sum_{\mathcal{T} \text{an } i \text{-tree}}{\bk}^\mathcal{T}$ is defined in Notation \ref{not:K}. Then $\bx$ is a complex-balanced equilibrium for the reaction rate vector ${\bk}$ if and only if 
Equations (\ref{eq:binom}) are satisfied for every pair of vertices in the same connected component in $G$.
\end{proposition}

\begin{proof}
From \eqref{def: Ak}, we get $[A_{\bk}]_{ji} \neq 0$, if $\byi \rightarrow \byj \in E$ or $i = j$. After we relabel the vertices according to the connected components of $G$, the Kirchoff matrix $A_{{\bk}}$ will be a block diagonal matrix, where each diagonal block corresponds to a connected component of $G$. 

Following Equations (\ref{eq:psiEquiv}), $\bx$ is a complex-balanced equilibrium if and only if 
$A_{\bk}\cdot \Psi(\bx) = \mathbf{0}$ under the reaction rate vector ${\bk}$.
Since we consider $A_{{\bk}}$ as a block diagonal matrix, it suffices to prove the proposition when the system has a single connected component (i.e. $\ell = 1$).

Now suppose $G = (V, E)$ has one connected component, thus it is strongly connected. Applying Lemma \ref{thm:Ak kernel} on the system $(G,\bk)$, we deduce that
\begin{equation} \label{eq: kerAk}
\dim (\text{ker}(A_{\bk})) = 1, \ \text{and } \det (A_{\bk}) = 0.
\end{equation}
Note that the minor of matrix $A_{\bk}$ is independent of the choice of rows because the column sums of $A_{\bk}$ are zero.
Using Proposition \ref{prop:minors} and expanding 
the determinant of $A_{\bk}$ in terms of its minors, we derive that 
\begin{equation}
A_{\bk}\cdot {\bK} = \mathbf{0},
\end{equation}
where ${\bK} = (K_1, K_2, \ldots, K_m)^\intercal$.

Now we obtain both ${\bK}$ and $\Psi(\bx)$ belongs to the null-space of $A_{\bk}$. One can check that they are both positive vectors.
From $\dim (\text{ker}(A_{\bk})) = 1$ in Equation \eqref{eq: kerAk}, we deduce that
the two vectors ${\bK}$ and $\Psi(\bx)$ are proportional. Hence, it is clear that $A_{\bk}\cdot \Psi(\bx) = \mathbf{0}$ if and only if Equations (\ref{eq:binom}) are satisfied for every pair of vertices of $G$.
Again using Equations (\ref{eq:psiEquiv}), we conclude this proposition.
\end{proof}

\begin{example}[See also {\cite[Equation 3.12]{MR339688}}] \label{ex:3vertices}
Consider a strongly connected mass-action system $(G, \bk)$ in Figure \ref{fig:3vertices}, with three vertices:
\begin{equation} \notag
{\boldsymbol{y}_1}=\begin{pmatrix} 2 \\ 0 \end{pmatrix}, \ \
{\boldsymbol{y}_2}=\begin{pmatrix} 1 \\ 1 \end{pmatrix}, \ \
{\boldsymbol{y}_3}=\begin{pmatrix} 0 \\ 2 \end{pmatrix}.
\end{equation} 

\begin{figure}[H]
    \centering
    \begin{tikzpicture}[scale=1.2]
    \draw[ultra thin,lightgray](-0.1,-0.1) grid (2.6,2.4);
    \draw[->,semithick] (-0.2,0) -- (2.8,0) node[right] {$X_1$};
    \draw[->,semithick] (0,-0.2) -- (0,2.6) node[above] {$X_2$};
    \node[label=right:{$\boldsymbol{y}_1$}](A) at (1.9,-.25) {};
    \node[label=right:{}](A) at (1.9,0) {};
    \node[label=right:{$\boldsymbol{y}_2$}] at (0.9,1.2) {};
    \node[label=right:{}](B) at (0.9,1) {};
    \node[label=right:{$\boldsymbol{y}_3$}] at (-0.1,2.2) {};
    \node[label=right:{}](C) at (-0.1,2) {};
    \draw[fill] (A) circle (1.5pt);
    \draw[fill] (B) circle (1.5pt);
    \draw[fill] (C) circle (1.5pt);

    \def\shf{0.2ex};
    \draw[->,transform canvas={xshift=\shf,yshift=\shf}, thick, red] (A) -- (B);
    \draw[->,transform canvas={xshift=-\shf,yshift=-\shf}, thick, red] (B) -- (A);
    \draw[->,transform canvas={xshift=\shf,yshift=\shf}, thick, red] (B) -- (C);
    \draw[->,transform canvas={xshift=-\shf,yshift=-\shf}, thick, red] (C) -- (B);

    \def\shf{0.6ex};
    \draw[->,transform canvas={xshift=\shf,yshift=\shf},thick, red] (C) -- (A);
    \draw[->,transform canvas={xshift=-\shf,yshift=-\shf},thick, red] (A) -- (C);
  \end{tikzpicture}
    \caption{Complete bidirected graph with three vertices, considered in Example \ref{ex:3vertices}.}
    \label{fig:3vertices}
\end{figure}

For the vertex $\by_1$, we list all spanning $\by_1$-trees of $G$ as follows:
\begin{figure}[H]
	\centering
     \begin{subfigure}[b]{0.3\textwidth}
         \centering
         \begin{tikzpicture}[scale=1.2]
    \draw[ultra thin,lightgray](-0.1,-0.1) grid (2.6,2.4);
    \draw[->,semithick] (-0.2,0) -- (2.8,0) node[right] {$X_1$};
    \draw[->,semithick] (0,-0.2) -- (0,2.6) node[above] {$X_2$};
    \node[label=right:{$\boldsymbol{y}_1$}](A) at (1.9,-.25) {};
    \node[label=right:{}](A) at (1.9,0) {};
    \node[label=right:{$\boldsymbol{y}_2$}] at (0.9,1.2) {};
    \node[label=right:{}](B) at (0.9,1) {};
    \node[label=right:{$\boldsymbol{y}_3$}] at (-0.1,2.2) {};
    \node[label=right:{}](C) at (-0.1,2) {};
    \draw[fill] (A) circle (1.5pt);
    \draw[fill] (B) circle (1.5pt);
    \draw[fill] (C) circle (1.5pt);

    \def\shf{0.2ex};
    \draw[->,transform canvas={xshift=-\shf,yshift=-\shf}, thick, red] (B) -- (A);
    \draw[->,transform canvas={xshift=-\shf,yshift=-\shf}, thick, red] (C) -- (B);
  \end{tikzpicture}
         \caption{}
         \label{a}
     \end{subfigure}
     \hfill
     \begin{subfigure}[b]{0.3\textwidth}
         \centering
         \begin{tikzpicture}[scale=1.2]
    \draw[ultra thin,lightgray](-0.1,-0.1) grid (2.6,2.4);
    \draw[->,semithick] (-0.2,0) -- (2.8,0) node[right] {$X_1$};
    \draw[->,semithick] (0,-0.2) -- (0,2.6) node[above] {$X_2$};
    \node[label=right:{$\boldsymbol{y}_1$}](A) at (1.9,-.25) {};
    \node[label=right:{}](A) at (1.9,0) {};
    \node[label=right:{$\boldsymbol{y}_2$}] at (0.9,1.2) {};
    \node[label=right:{}](B) at (0.9,1) {};
    \node[label=right:{$\boldsymbol{y}_3$}] at (-0.1,2.2) {};
    \node[label=right:{}](C) at (-0.1,2) {};
    \draw[fill] (A) circle (1.5pt);
    \draw[fill] (B) circle (1.5pt);
    \draw[fill] (C) circle (1.5pt);

    \def\shf{0.2ex};
    
    \draw[->,transform canvas={xshift=-\shf,yshift=-\shf}, thick, red] (B) -- (A);

    \def\shf{0.6ex};
    \draw[->,transform canvas={xshift=\shf,yshift=\shf},thick, red] (C) -- (A);
    
  \end{tikzpicture}
         \caption{}
         \label{b}
     \end{subfigure}
     \hfill
     \begin{subfigure}[b]{0.3\textwidth}
         \centering
         \begin{tikzpicture}[scale=1.2]
    \draw[ultra thin,lightgray](-0.1,-0.1) grid (2.6,2.4);
    \draw[->,semithick] (-0.2,0) -- (2.8,0) node[right] {$X_1$};
    \draw[->,semithick] (0,-0.2) -- (0,2.6) node[above] {$X_2$};
    \node[label=right:{$\boldsymbol{y}_1$}](A) at (1.9,-.25) {};
    \node[label=right:{}](A) at (1.9,0) {};
    \node[label=right:{$\boldsymbol{y}_2$}] at (0.9,1.2) {};
    \node[label=right:{}](B) at (0.9,1) {};
    \node[label=right:{$\boldsymbol{y}_3$}] at (-0.1,2.2) {};
    \node[label=right:{}](C) at (-0.1,2) {};
    \draw[fill] (A) circle (1.5pt);
    \draw[fill] (B) circle (1.5pt);
    \draw[fill] (C) circle (1.5pt);

    \def\shf{0.2ex};
    
    \draw[->,transform canvas={xshift=-\shf,yshift=-\shf}, thick, red] (B) -- (C);

    \def\shf{0.6ex};
    \draw[->,transform canvas={xshift=\shf,yshift=\shf},thick, red] (C) -- (A);
    
  \end{tikzpicture}
         \caption{}
         \label{c}
     \end{subfigure}
    \caption{Spanning $\by_1$-trees of $G$.}
    \label{fig:spanning tree in 3vertices}
\end{figure}

From Notation \ref{not:K}, we obtain that
\begin{equation} \notag
K_1 = k_{21}k_{31}+k_{32}k_{21}+k_{23}k_{31}. 
\end{equation}
Analogously, we can derive $K_2, K_3$ corresponding to the vertices $\by_2, \by_3$ in $G$,
\begin{equation} \notag
\begin{split}
& K_2 = k_{12}k_{32}+k_{13}k_{32}+k_{31}k_{12}
\\& K_3 = k_{13}k_{23}+k_{21}k_{13}+k_{12}k_{23}.
\end{split}
\end{equation}

Suppose $\bx = (x_1, x_2)$ is a complex-balanced steady state. Using Proposition \ref{prop:iff}, we get that $\bk \in \mathcal{V}(G)$, if and only if
\begin{equation} \label{eq: cb in 3.10 1}
\frac{K_1}{\bx^{{\boldsymbol{y}_1}}}
= \frac{K_2}{\bx^{{\boldsymbol{y}_2}}}
= \frac{K_3}{\bx^{{\boldsymbol{y}_3}}}.
\end{equation}
By eliminating $x_1, x_2$ in equation \eqref{eq: cb in 3.10 1}, the moduli space $\mathcal{V}(G)\subset \mathbb{R}^6_{>0}$ must satisfy the following binomial:
\begin{equation} \label{eq: cb in 3.10 2}
K_1K_3 - K_2^2 = 0,
\end{equation}
Therefore, we recover the result from \cite[Equation 3.12]{MR339688} (see also \cite[Example 1]{CDSS}, \cite[page 195]{MR4284895}): the toric locus (moduli space) can be written as 
\begin{equation}
\begin{split} \notag
\mathcal{V}(G) = \big\{ {\bk}\in\mathbb{R}_{>0}^6: \
(k_{21}k_{31}+k_{32}k_{21}+k_{23}k_{31}) 
 (k_{13}k_{23}+k_{21}k_{13}+ k_{12}k_{23}) &
\\ -(k_{12}k_{32}+k_{13}k_{32}+k_{31}k_{12})^2 & = 0 \big\}.
\end{split}
\end{equation}

\end{example}

\begin{definition}[\cite{shastri}]
Consider two manifolds $A$ and $B$ in the  Euclidean space $\RR^n$. We say that $A$ and $B$   \defi{intersect transversally}, if at any intersection point $x \in A \cap B$, $T_x (A) + T_x (B) = \RR^n$, that is, their tangent spaces span $\RR^n$. 
\end{definition}

Before we proceed to the proof of Theorem \ref{th:qCont} we also need the following lemma:

\begin{lemma}[{\cite[Lemma 5.4]{MR4034297}}]
\label{thm:transversal} 
 Let $\bx_1, \bx_2 \in \RR^n_{>0}$ be positive vectors.  Consider  a vector subspace $S$ in $\RR^n$. Let $\bx_1 + S$ and $\bx_2 \circ \exp (S^\perp)$ be two manifolds of $\RR^n$. Then the two manifolds intersect transversally, i.e., 
\begin{equation} \notag
T_{\bp} (\bx_1 + S) + T_{\bp} (\bx_2 \circ \exp (S^\perp)) = \RR^n,
\end{equation}
for any point $\bp \in (\bx_1 + S) \cap (\bx_2 \circ \exp (S^\perp))$.
\end{lemma}

Finally, we are prepared to prove Theorem \ref{th:qCont}. 
Let us roughly explain the main ideas of the proof. By Notation \ref{not:K}, we have \begin{equation} \notag
K_i := \sum_{\mathcal{T} \text{an } i \text{-tree}}{\bk}^\mathcal{T}.
\end{equation} We will show that the set of complex-balanced equilibria depends continuously on ${\bK}$. There are two main steps. First, we prove the theorem in the case where the graph $G$ has only one connected component; second, we generalize the result for any number of connected components. 
In the case of one connected component, we proceed as follows. For any state ${\boldsymbol{x}_0} \in \mathbb{R}_{>0}^n$, the corresponding complex-balanced equilibrium is the unique intersection between the set of complex-balanced equilibria and the affine invariant polyhedron $({\boldsymbol{x}_0} + \mathcal{S}) \cap \mathbb{R}_{>0}^{n}$. We find a vector $\bX^* \in \mathcal{S}$, such that $\exp  (\bX^*)$ is a complex-balanced equilibrium of the system $(G, \bk)$.  
By Lemma \ref{thm:transversal} we have that  
$({\boldsymbol{x}_0} + \mathcal{S})$ and $\exp (\bX^* + \mathcal{S}^{\perp})$ intersect transversally, thus the unique intersection point varies continuously as a function of $\bX^*$. 
Since $\bX^*$ depends continuously on ${\bK}$, we conclude the proof for the case of one connected component, and then extend to the general case.

\begin{proof}[\textbf{Proof of Theorem \ref{th:qCont}}]

Here, for the sake of simplicity, we temporarily make the following abuse of notation:
\begin{equation} \label{not:X}
\bX = (X_1, \cdots, X_n)^\intercal := \ln \bx=(\ln x_1,\ldots, \ln x_n)^\intercal.
\end{equation}

Recall Notation \ref{not:K}, for each vertex $\byi \in V$, we have
\begin{equation} \notag
K_i = \sum_{\mathcal{T} \text{an } i \text{-tree}}{\bk}^\mathcal{T},
\end{equation}
where ${\bk}^\mathcal{T}$ is the product of reaction rates $k_{ij}$ associated with reactions in the spanning $\by_i$-tree $\mathcal{T}$ of $G$.
It is standard to derive that the vector ${\bK} = (K_i)  \in \mathbb{R}_{>0}^m$ depends smoothly 
on the reaction rate vector ${\bk} = (k_{ij})  \in \mathbb{R}_{>0}^{E}$. Hence, it suffices for us to show that the set of complex-balanced equilibria depends continuously on ${\bK}$.

By Proposition \ref{prop:iff}, a state $\bx$ is a complex-balanced equilibrium if and only if for any two vertices $\byi, \byj$ in the same connected component of $G$,
\begin{equation} \label{eq:binom2}
K_i \bx^{\byj}
= K_j \bx^{\byi}.
\end{equation}
Taking the log of both sides in Equation \eqref{eq:binom2}, we derive
\begin{equation} \label{eq:logMultivar1}
\ln (K_i) + \byj^\intercal \cdot \ln (\bx)
=\ln (K_j) + \byi^\intercal \cdot \ln (\bx).
\end{equation} 
Thus, we can rewrite (\ref{eq:logMultivar1}) as
\begin{equation}\label{eq:logMultivar2}
\ln (K_i / K_j)
= (\byi^\intercal - \byj^\intercal) \cdot \bX,
\end{equation}
where $\byi$ and $\byj$ are two vertices belonging to the same connected component of $G$.

We show the rest of the proof in two steps. 
First, we prove the theorem under the assumption that the graph $G$ has only one connected component.
Next, we explain how to generalize the result into an arbitrary number of connected components. 

Now suppose the graph $G$ has a single connected component (i.e. $\ell = 1$), then all vertices $\{{\boldsymbol{y}_1}, \ldots, {\boldsymbol{y}_m}\}$ are in the same connected component. It is clear that Equations (\ref{eq:logMultivar2}) are equivalent to the following system of linear equations in $\bX$:
\begin{equation} \label{eq:logMultivar4}
\begin{bmatrix}
\ln ( K_1 / K_2) \\
\ln ( K_2 / K_3) \\
\vdots \\
\ln ( K_{m-1} / K_m)
\end{bmatrix}
=\begin{bmatrix}
 \by_1^\intercal - \by_2^\intercal \\
\by_2^\intercal - \by_3^\intercal  \\
\vdots \\
\by_{m-1}^\intercal - \by_{m}^\intercal
\end{bmatrix}
\begin{bmatrix}
 X_1\\
 X_2\\
\vdots \\
X_n
\end{bmatrix}.
\end{equation}
After we set
\begin{equation} \notag
\Delta \boldsymbol{y} := 
\begin{bmatrix}
 \by_1^\intercal - \by_2^\intercal \\
\by_2^\intercal - \by_3^\intercal  \\
\vdots \\
\by_{m-1}^\intercal - \by_{m}^\intercal
\end{bmatrix}, \ \text{and }
\Delta {\bK} :=
\begin{bmatrix}
K_1 / K_2 \\
K_2 / K_3 \\
\vdots \\
K_{m-1}/ K_m
\end{bmatrix},
\end{equation} 
the system (\ref{eq:logMultivar4}) can be expressed as 
\begin{equation} \label{system 1}
\ln (\Delta {\bK})=(\Delta \boldsymbol{y}) \bX.
\end{equation}
Since $G$ is strongly connected, then its stoichiometric subspace is 
\begin{equation} \notag
\mathcal{S} = \spn \{ \by_1^\intercal -\by_2^\intercal, \by_2^\intercal -\by_3^\intercal, \ldots, \by_{m-1}^\intercal -\by_{m}^\intercal \}.
\end{equation}
Let $s$ be the dimension of  $\mathcal{S}$, then we deduce that $s \leq \min \{ m-1, n\}$, and the matrix $\Delta \boldsymbol{y}$ has exactly $s$ linearly independent rows.
W.l.o.g. we assume the first $s$ rows in $\Delta \boldsymbol{y}$ are linearly independent. Thus, we obtain
\begin{equation} \label{S span}
\mathcal{S} = \spn \{ \by_1^\intercal -\by_2^\intercal, \by_2^\intercal -\by_3^\intercal, \ldots, \by_{s}^\intercal -\by_{s+1}^\intercal \}.
\end{equation}
Furthermore, we consider the system of equations as follows:
\begin{equation} \label{system 2}
\ln (\Delta_s {\bK})=(\Delta_s \boldsymbol{y}) \bX,
\end{equation}
where
\begin{equation} \notag
\Delta_s \boldsymbol{y} := 
\begin{bmatrix}
 \by_1^\intercal - \by_2^\intercal \\
\by_2^\intercal - \by_3^\intercal  \\
\vdots \\
\by_s^\intercal - \by_{s+1}^\intercal
\end{bmatrix}, \ \text{and }
\Delta_s {\bK}:=\begin{bmatrix}
K_1 / K_2 \\
K_2 / K_3 \\
\vdots \\
K_{s} / K_{s+1}
\end{bmatrix}.
\end{equation}

Since $\bk \in \mathcal{V}(G)$, by Theorem \ref{thm:HJ}, the complex-balanced system $(G, \bk)$ must admit one complex-balanced steady state $\bx^* \in \RR^n_{>0}$, i.e. $\ln \bx^*$ is a solution to \eqref{system 1}.
From Theorem \ref{thm:HJ}, any complex-balanced steady state $\bx$ satisfies $\ln \bx - \ln \bx^* \in \mathcal{S}^\perp$, where $\mathcal{S}^{\perp}$ denotes the orthogonal complement of $\mathcal{S}$. Thus the solutions to \eqref{system 1} can be written as 
$\bX = \ln \bx^* + \mathcal{S}^\perp$, and this shows the dimension of the set of solutions to (\ref{system 1}) is $n-s$. 

Moreover, it is straight to check that the solutions of (\ref{system 1}) must solve (\ref{system 2}).
Since the rows in the matrix $\Delta_s \boldsymbol{y}$ are linearly independent, the set of solutions to (\ref{system 2}) is also of dimension $n-s$. Therefore, we conclude that system (\ref{system 1})  is equivalent to system (\ref{system 2}) in solving $\bX$.

Next, we construct a special solution $\bX^*$ to system \eqref{system 2} with $\bX^* \in \mathcal{S}$.
Recall that $s \leq \min \{ m-1, n\}$. In function of $s$ the dimension  of the stoichiometric subspace, we consider two cases:

\medskip

\textbf{Case 1: } $s = n$. Then the stoichiometric subspace $\mathcal{S} = \mathbb{R}^n$, and $\Delta_s \boldsymbol{y} \in \mathbb{R}_{n \times n}$ is a square full rank matrix, i.e. $\Delta_s \boldsymbol{y}$ is invertible.
Thus, we derive a solution of (\ref{system 2}) as
\begin{equation} \label{X* 1}
\bX^* = (\Delta_s \boldsymbol{y})^{-1} \ln (\Delta_s {\bK}). 
\end{equation}
It is clear that $\bX^* \in \mathcal{S} = \mathbb{R}^n$, and $\exp  (\bX^*)$ satisfies equations  \eqref{eq:binom2} by construction. This ensures that $\exp  (\bX^*)$ is a complex-balanced equilibrium.

\medskip

\textbf{Case 2: } $s < n$.
Recall that $\mathcal{S}^{\perp}$ denotes the orthogonal complement of $\mathcal{S}$. 
Since the stoichiometric subspace $\mathcal{S} \subset \mathbb{R}^n$, we obtain that $\mathcal{S}^{\perp} \neq \emptyset$ and
\begin{equation} \label{S perp}
0 < \dim (\mathcal{S}^{\perp}) = n - \dim (\mathcal{S}) = n - s.
\end{equation}
Then we consider a basis of $\mathcal{S}^{\perp}$, denoted by $B$, such that  
\begin{equation} \notag
B = \{ \bv_1, \bv_2, \ldots, \bv_{n-s} \} \subset \mathbb{R}^n.
\end{equation}
Furthermore, we build another matrix and vector below
\begin{equation} \label{tilde matirx}
\tilde{\Delta} \boldsymbol{y} := 
\begin{bmatrix}
\Delta_s \boldsymbol{y} \\
\bv_1^\intercal \\
\vdots \\
\bv_{n-s}^\intercal
\end{bmatrix}, \ \text{and }
\tilde{\Delta} {\bK} := \begin{bmatrix}
\Delta_s {\bK} \\
0 \\
\vdots \\
0
\end{bmatrix},
\end{equation}
and consider the following system:
\begin{equation} \label{system 3}
\ln (\tilde{\Delta} {\bK}) = (\tilde{\Delta} \boldsymbol{y}) \bX.
\end{equation}
It is clear that the solutions of (\ref{system 3}) must solve (\ref{system 2}). 
From \eqref{S span} and $\{ \bv_1, \ldots, \bv_{n-s} \}$ forming a basis of $\mathcal{S}^{\perp}$, we deduce that $\tilde{\Delta} \boldsymbol{y} \in \mathbb{R}_{n \times n}$ is an invertible matrix.
Hence, we obtain a solution of (\ref{system 3}) as
\begin{equation} \label{X* 2}
	\bX^* = (\tilde{\Delta} \boldsymbol{y})^{-1} \ln (\tilde{\Delta} {\bK}).
\end{equation}
Moreover, for $i = 1, \cdots, n-s$, we have
\begin{equation} \notag
\bv_i^\intercal \cdot \bX^*  = 0,
\end{equation}
and this shows that $\bX^* \in \mathcal{S}$. 
By construction, $\exp  (\bX^*)$ must solve equations  \eqref{eq:binom2}, thus it is a complex-balanced equilibrium. 

In conclusion, we have found a vector $\bX^* \in \mathcal{S}$, such that $\exp  (\bX^*)$ is a complex-balanced equilibrium of the system $(G, \bk)$ in both cases.
Further, using the fact that both $(\Delta_s \boldsymbol{y})^{-1}$ and $(\tilde{\Delta} \boldsymbol{y})^{-1}$ are fixed real matrices, we deduce $\bX^*$ depends smoothly  
on the vector ${\bK}$.
From Theorem \ref{thm:HJ}(b), given a complex-balanced system $(G, \bk)$ and one complex-balanced steady state  $\exp  (\bX^*)$ constructed above, the set of all complex-balanced equilibria of the system can be written as $\exp (\bX^* + \mathcal{S}^{\perp})$.

More specifically, for any state ${\boldsymbol{x}_0} \in \mathbb{R}_{>0}^n$, the corresponding complex-balanced equilibrium is the unique intersection between the set of complex-balanced equilibria 
$\exp (\bX^* + \mathcal{S}^{\perp})$ and the affine invariant polyhedron $({\boldsymbol{x}_0} + \mathcal{S}) \cap \mathbb{R}_{>0}^{n}$.
Using Lemma \ref{thm:transversal}, we get that the two manifolds $({\boldsymbol{x}_0} + \mathcal{S})$ and $\exp (\bX^* + \mathcal{S}^{\perp})$ intersect transversally. 
Hence, given a state ${\boldsymbol{x}_0}$, the (unique) intersection point varies continuously as a function of $\bX^*$. 
Together with the fact that $\bX^*$ depends continuously on ${\bK}$, which additionally varies continuously on ${\bk}$, we conclude that the map $Q_{{\boldsymbol{x}_0}}$ is continuous on $\bk \in \mathcal{V}(G)$ when the graph $G$ has only one connected component. 

Finally, we consider the case when the graph $G$ has multiple connected components, $V_1, \ldots, V_\ell$ with $\ell > 1$.
Following the proof in Proposition \ref{prop:iff}, we can relabel the vertices according to the connected components of $G$, i.e., for $1 \leq p \leq \ell$,
\[
V_p = \{ \by_{m_{p-1}+1}, \ldots, \by_{m_p}\},
\]
such that the Kirchoff matrix $A_{{\bk}}$ will be a block diagonal matrix, where each diagonal block corresponds to a connected component of $G$. 

Recall from Equations \eqref{eq:logMultivar2}, a state $\bx$ is a complex-balanced equilibrium, if and only if for any two vertices $\byi, \byj$ in the same connected component of $G$,
\begin{equation} \notag
\ln (K_i / K_j)
= (\byi^\intercal - \byj^\intercal) \cdot \bX,
\end{equation}
which is equivalent to the following system of linear equations in $\bX$:
\begin{equation} \label{eq:logMultivar4 l>1}
\underbrace{
\begin{bmatrix}
\ln ( K_1 / K_2) \\
\vdots \\
\ln ( K_{m_1 - 1} / K_{m_1}) 
\\
\hdashline[1.5pt/1.5pt]
\ln ( K_{m_1 + 1} / K_{m_1 + 2}) \\
\vdots \\
\ln ( K_{m_2 - 1} / K_{m_2}) 
\\
\hdashline[1.5pt/1.5pt]
\vdots \\
\ln ( K_{m_{\ell}-1} / K_{m_{\ell}})
\end{bmatrix} 
}_{\ln (\Delta {\bK})}
= 
\underbrace{
\begin{bmatrix}
 \by_1^\intercal - \by_2^\intercal \\
\vdots \\
\by_{m_1 - 1}^\intercal - \by_{m_1}^\intercal 
\\
\hdashline[1.5pt/1.5pt]
\by_{m_1 + 1}^\intercal - \by_{m_1 + 2}^\intercal 
\\
\vdots \\
\by_{m_2 - 1}^\intercal - \by_{m_2}^\intercal  
\\
\hdashline[1.5pt/1.5pt]
\vdots \\
\by_{m_{\ell}-1}^\intercal - \by_{m_{\ell}}^\intercal
\end{bmatrix}
}_{\Delta \boldsymbol{y}}
\begin{bmatrix}
X_1 \\
X_2 \\
\vdots \\
X_n
\end{bmatrix},
\end{equation}
and we can express it as
\begin{equation} \label{system 1 l>1}
\ln (\Delta {\bK}) = (\Delta \boldsymbol{y}) \bX.
\end{equation}
Since $G$ has $\ell$ connected components, its stoichiometric subspace is 
\begin{equation} \notag
\mathcal{S} = \spn \{ \by_1^\intercal -\by_2^\intercal, \ldots, \by_{m_1 - 1}^\intercal - \by_{m_1}^\intercal, 
\by_{m_1 + 1}^\intercal - \by_{m_1 + 2}^\intercal, \ldots, \by_{m_{\ell}-1}^\intercal -\by_{m_{\ell}}^\intercal \}.
\end{equation}
Let $s$ be the dimension of  $\mathcal{S}$. Then we deduce that $s \leq \min \{ m-\ell, n\}$, and the matrix $\Delta \boldsymbol{y}$ has exactly $s$ linearly independent rows.

Analogously, we pick $s$ linear independent rows in $\Delta \boldsymbol{y}$, and they also span stoichiometric subspace $\mathcal{S}$.
Moreover, these rows in 
$\Delta \boldsymbol{y}$ formulate a full row rank matrix $\Delta_s \boldsymbol{y}$, while the corresponding rows in $\ln (\Delta {\bK})$ gives us the vector $\ln (\Delta_s {\bK})$.
And it is easy to check that system (\ref{system 1 l>1}) is equivalent to the following system in $\bX$:
\begin{equation} \label{system 2 l>1}
\ln (\Delta_s {\bK})=(\Delta_s \boldsymbol{y}) \bX.
\end{equation}

Next, we construct a special solution $\bX^*$ to system \eqref{system 2 l>1} with $\bX^* \in \mathcal{S}$.
Similarly, we consider $s$ the dimension of the stoichiometric subspace in two cases: $s = n$ and $s < n$.

If $s = n$, then $\Delta_s \boldsymbol{y}$ is an invertible matrix.
Thus, we derive a solution of (\ref{system 2 l>1}) as
\begin{equation} \notag
\bX^* = (\Delta_s \boldsymbol{y})^{-1} \ln (\Delta_s {\bK}). 
\end{equation}
It is easy to see that $\bX^* \in \mathcal{S} = \mathbb{R}^n$, and $\exp  (\bX^*)$ is a complex-balanced equilibrium.

If $s < n$, we obtain
$\dim (\mathcal{S}^{\perp}) = n - s > 0$, and consider a basis $B = \{ \bv_1, \ldots, \bv_{n-s} \}$ of $\mathcal{S}^{\perp}$.
Similar as in Equations \eqref{tilde matirx}-\eqref{X* 2}, we first add $\bv_1^\intercal, \ldots, \bv_{n-s}^\intercal$ on the bottom of the matrix $\Delta_s \boldsymbol{y}$, and adapt $n-s$ zeros to the vector $\ln (\Delta_s {\bK})$.
Then, we obtain the desired solution $\bX^* \in \mathcal{S}$ of \eqref{system 2 l>1}, with $\exp  (\bX^*)$ is a complex-balanced equilibrium.

Together with both cases, we deduce $\bX^*$ depends smoothly on the vector ${\bK}$. We omit the rest of the proof since it straightly follows from the single connected component case.
\end{proof}

The following is a direct consequence of results within the proof of the Theorem \ref{th:qCont}.
\begin{corollary}
Let $G = (V, E)$ be a weakly reversible E-graph with the stoichiometric subspace $\mathcal{S}$. For any $\bk \in \mathcal{V}(G)$, there exists a unique complex-balanced equilibrium $\bx^*$, such that $\ln (\bx^*) \in \mathcal{S}$ and $\bx^*$ depends smoothly on the parameter values  $\bk \in \mathcal{V}(G)$.
\end{corollary}

\begin{definition}[\cite{lee}]
A surjective, continuous, and open map is called a \defi{quotient map}.
\end{definition}

\begin{corollary} \label{cor:quot}
For any state ${\boldsymbol{x}_0} \in \mathbb{R}_{>0}^n$,
the map $Q_{{\boldsymbol{x}_0}}$ from Definition \ref{def:q} is a quotient map.
\end{corollary}

\begin{proof}
From Lemma \ref{lemma:Qsurj} and Lemma \ref{lemma:Qopen}, we proved the map $Q_{{\boldsymbol{x}_0}}$ is surjective and open. 
Together with Theorem \ref{th:qCont}, we conclude the map $Q_{{\boldsymbol{x}_0}}$ is a quotient map. 
\end{proof}

\subsection{The toric locus \texorpdfstring{$\mathcal{V}(G)$}{VG} is connected}
\label{sec:Conn}

The main result of this section is Theorem \ref{th:main}, where we show the connectedness of the toric locus $\mathcal{V}(G)$. 
We first recall a fundamental result in general topology as follows:

\begin{lemma}[{\cite[Theorem 9.4]{MR0264581}}]\label{lemma:threeTopSp}
Consider three topological spaces $A, B, C$ and a surjective map $f: A \to B$. Let B be endowed with the quotient topology induced by $f$. 
Given an arbitrary map $g: B \to C$, then $g$ is continuous if and only if the map $g \circ f: A \to C$ is continuous.
\end{lemma}

\begin{theorem} \label{th:main}
Let $G = (V, E)$ be a weakly reversible E-graph. Then the toric locus $\mathcal{V}(G)$ is connected.
\end{theorem}

\begin{proof}
We will argue by contradiction. 
Given a state ${\boldsymbol{x}_0} \in \mathbb{R}_{>0}^n$, suppose the set $\mathcal{V}(G)$ is not connected. Then there exists a \emph{surjective continuous} map $\mu$, such that
\begin{equation*} 
    \mu :\mathcal{V}(G)\rightarrow \{0,1\}.
\end{equation*}
Next we consider the following commutative diagram:
\begin{figure}[H]
\centering
\begin{tikzcd}
\mathcal{V}(G) \arrow{d}[left = 5pt]{Q_{{\boldsymbol{x}_0}}} \arrow{r}[above = -5pt]{\mu} & \{0,1\} \\[10pt]
({\boldsymbol{x}_0}+\mathcal{S})\cap \mathbb{R}_{>0}^n \arrow{ur}[right = 10pt]{\nu} & 
\end{tikzcd}
\end{figure}

\noindent The map $\nu:({\boldsymbol{x}_0}+\mathcal{S})\cap \mathbb{R}_{>0}^n\rightarrow \{0,1\}$ in the diagram satisfies 
\begin{equation*}
    \mu=\nu\circ Q_{{\boldsymbol{x}_0}}.
\end{equation*}
It is well-defined since $Q_{{\boldsymbol{x}_0}}$ is surjective by Lemma \ref{lemma:Qsurj}. 

By Corollary \ref{cor:quot}, the map $Q_{{\boldsymbol{x}_0}}$ is a quotient map. Hence, by Lemma \ref{lemma:threeTopSp} we derive that $\nu$ is continuous if and only if $\mu$ is continuous. 
Since $\mu$ is continuous, we conclude that $\nu$ is a continuous map.
We also derive that $\nu$ is surjective from $\mu$ being a subjective map.

Note that the invariant polyhedron $({\boldsymbol{x}_0}+\mathcal{S})\cap \mathbb{R}_{>0}^n$ is connected, while the set $\{0,1\}$ is clearly disconnected.
This leads to a contradiction since every continuous function maps a connected set to a connected set.
Thus the initial supposition is false, and we conclude that $\mathcal{V}(G)$ is connected.
\end{proof}

\section{The toric locus \texorpdfstring{$\mathcal{V}(G)$}{Vg} is a product space}
\label{sec:product}

In this section, we consider a weakly reversible E-graph $G=(V, E)$ and show the toric locus $\mathcal{V}(G)$ is a product space in Theorem \ref{th:product}.

\subsection{The set of complex-balanced flux vectors \texorpdfstring{$\mathcal{B}(G)$}{Bg}}

\begin{definition}
\label{def:flux system}
Given an E-graph $G=(V, E)$, we let $\bbeta = (\beta_{\byi \to \byj})_{\byi \to \byj \in E} \in \RR_{>0}^E$ denote a  \defi{flux vector}, where the component $\beta_{\byi \to \byj} > 0$ is called the \defi{flux} of the reaction $\byi \to \byj$. 
Moreover, the pair $(G, \bbeta)$ is called a \defi{flux system}.
\end{definition}

\begin{definition}
\label{def:fluxVectors}

Consider an E-graph $G=(V,E)$, a flux vector $\bbeta \in \RR_{>0}^E$ is called a \defi{steady flux vector} on $G$ if
\begin{equation}
\label{def: SSF}
\sum_{\byi \to \byj \in E} \beta_{\byi \to \byj} 
(\byj - \byi) = \mathbf{0}.
\end{equation}
A steady flux vector $\bbeta$ is called a \defi{complex-balanced flux vector} if at each vertex $\by_0 \in V$, 
\begin{equation}
\label{eq:beta}
\sum_{\by \to \by_0 \in E} \beta_{\by \to \by_0} 
= \sum_{\by_0 \to \byp \in E} \beta_{\by_0 \to \byp},
\end{equation} 
and we say that the pair $(G, \bbeta)$ is a \defi{complex-balanced flux system}.
\end{definition}

\begin{definition}
Given an E-graph $G=(V,E)$, we define the \defi{set of complex-balanced flux vectors} on $G$ as follows:
\begin{equation} 
\label{def: BG}
\mathcal{B}(G) :=
\{\bbeta \in \RR_{>0}^{E} \mid \bbeta  \text{ is a complex-balanced flux vector on $G$}\}.
\end{equation}
\end{definition}

Analogous to complex-balanced mass action systems, complex-balanced flux systems also have connections with E-graphs.

\begin{lemma}
\label{lem:cbf connection with graph}

Every E-graph which permits a complex-balanced flux system is weakly reversible.
Moreover, every E-graph which is weakly reversible permits complex-balanced flux systems.
\end{lemma}

\begin{proof}
First, suppose the E-graph $G = (V, E)$ allows a complex-balanced flux system $\bbeta = (\beta_{\byi \to \byj})_{\byi \to \byj \in E} \in \RR_{>0}^E$. We define a mass-action system $(G, \bk)$ with reaction rate constants 
\begin{equation} \notag
k_{\by \to \byp} = \beta_{\by \to \byp}, \ \text{for every } \by \to \byp \in E.
\end{equation}
Then, it is clear that $\bx^* = (1, \ldots, 1)^T$ is a complex-balanced steady state. 
Applying Theorem \ref{thm:cb connection with graph}, we deduce that $G = (V, E)$ is weakly reversible.

\smallskip

Next, assume that the E-graph $G = (V, E)$ is weakly reversible. 
From Theorem \ref{thm:cb connection with graph}, there exists a complex-balanced mass action system $(G, \bk)$ with a steady state $\bx^*$. 
We define a flux system $(G, \bbeta)$ with fluxes
\begin{equation} \notag
\beta_{\by \to \byp} := k_{\by \to \byp} (\bx^*)^{\by}, \ \text{for every } \by \to \byp \in E.
\end{equation}
Inputting $\bbeta$ into \eqref{eq:beta}, we derive that $(G, \bbeta)$ is a complex-balanced flux system.
\end{proof}

Subsequently, given an E-graph $G=(V,E)$, we conclude that
\begin{itemize}
\item If $G=(V,E)$ is weakly reversible, then $\bB (G) \neq \emptyset$.

\item If $G=(V,E)$ isn't weakly reversible, then $\bB (G) = \emptyset$.
\end{itemize}
Since we are not interested in the case when $\bB (G)$ is empty, thus we always assume that the E-graph $G=(V,E)$ is weakly reversible when working on $\bB (G)$.

\begin{lemma} \label{lem:Bproduct}

Let $G = (V, E)$ be a weakly reversible E-graph. Then the set of complex-balanced flux vectors $\mathcal{B}(G)$ is a convex cone in $\mathbb{R}_{>0}^{E}$.
\end{lemma}

\begin{proof}
Suppose two flux vectors $\bbeta^*, \bbeta^{**} \in \bB (G)$, then we get
\begin{equation}
\sum_{\by \to \by_0 \in E} \beta^*_{\by \to \by_0} 
= \sum_{\byp \to \by \in E} \beta^*_{\byp \to \by}, \ \text{and } 
\sum_{\by \to \by_0 \in E} \beta^{**}_{\by \to \by_0} 
= \sum_{\byp \to \by \in E} \beta^{**}_{\byp \to \by}.
\end{equation}
Now we consider the following set:
\begin{equation}
L (\bbeta^*, \bbeta^{**})
:= \{ a \bbeta^* + (1-a) \bbeta^{**}: 0 \leq a \leq 1\}.
\end{equation}
Under direct computation, we obtain for any number $0 \leq a \leq 1$,
\begin{equation}
\sum_{\byi\rightarrow \byj}
(a \beta^{*}_{\byi \rightarrow \byj} + (1-a) \beta^{**}_{\byi \rightarrow \byj}) 
= \sum_{\byj\rightarrow \byi}(a \beta^{*}_{\byj \rightarrow \byi} + (1-a) \beta^{**}_{\byj \rightarrow \byi}).
\end{equation}
Therefore, $L (\bbeta^*, \bbeta^{**})  \subset \bB (G)$ and we prove this Lemma.
\end{proof}

The following remark is a direct consequence of Lemma \ref{lem:Bproduct}.

\begin{remark}
Consider a weakly reversible E-graph $G=(V,E)$. Then the set of complex-balanced flux vectors $\mathcal{B}(G)$ is connected.
\end{remark}

\subsection{The toric locus \texorpdfstring{$\mathcal{V}(G)$}{Vg} is a product space}

The goal of this section is to establish the \emph{product structure} of the moduli spaces of toric dynamical systems via an explicitly constructed homeomorphism. 

Let us recall the well-known properties of a homeomorphism (see for instance \cite{lee}):

\begin{definition}
\label{def:homeo}

A function $f : X \to Y$ between two topological spaces is a homeomorphism, if it has the following properties:
$f$ is bijective, continuous and the inverse function $f^{-1}$ is continuous.
If such a function $f$ exists, we say two topological spaces $X$ and $Y$ are homeomorphic, and write this as
$X \simeq Y$.
\end{definition}

Now we present the main result in this paper.

\begin{theorem}
\label{th:product}

Let $G = (V, E)$ be a weakly reversible E-graph.
For any state ${\boldsymbol{x}_0} \in \mathbb{R}_{>0}^n$,
the toric locus $\mathcal{V}(G) \subseteq \mathbb{R}_{>0}^{E}$ is homeomorphic to the product space $\mathcal{S}_{{\boldsymbol{x}_0}} \times \mathcal{B}(G)$, that is, 
\begin{equation}
\mathcal{V}(G)	\simeq 
\mathcal{S}_{{\boldsymbol{x}_0}} \times \mathcal{B}(G),
\end{equation}
where $\mathcal{S}_{{\boldsymbol{x}_0}}$ is the invariant polyhedron, and $\mathcal{B}(G)$ is the set of complex-balanced flux vectors.
\end{theorem}

\medskip

To prove Theorem \ref{th:product}, we start by constructing a function $\varphi$ between the product space $\mathcal{S}_{{\boldsymbol{x}_0}} \times \mathcal{B}(G)$ and the toric locus  $\mathcal{V}(G)$. Then we show that $\varphi$ is a homeomorphism.

\begin{definition}
\label{def: varphi}
Let $G = (V, E)$ be a weakly reversible E-graph.
Given a state ${\boldsymbol{x}_0} \in \mathbb{R}_{>0}^n$, we define the following map:
\begin{equation} 
\varphi: \mathcal{S}_{{\boldsymbol{x}_0}} \times \mathcal{B}(G)\rightarrow \mathcal{V}(G),
\end{equation}
such that for any $\bx \in \mathcal{S}_{{\boldsymbol{x}_0}}$ and $\bbeta = (\beta_{\byi \to \byj})_{\byi \to \byj \in E} \in \mathcal{B}(G)$, 
\begin{equation} \label{eq:phixb}
\varphi (\bx, \bbeta) := (\varphi_{\byi \to \byj})_{\byi \to \byj \in E}, \ \text{with } \         
\varphi_{\byi \to \byj} := \frac{\beta_{\byi \to \byj}}{\bx^{\byi}}.
\end{equation}
\end{definition}
\begin{lemma}
\label{lem:phiContin}

For any state ${\boldsymbol{x}_0} \in \mathbb{R}_{>0}^n$,
the map $\varphi$ is well-defined, and continuous.
\end{lemma}

\begin{proof}
For any $\bbeta \in \bB (G) \subseteq \mathbb{R}_{>0}^{E}$ and $\bx \in \mathcal{S}_{{\boldsymbol{x}_0}} \subseteq \mathbb{R}_{>0}^n$, we get 
\begin{equation}
\varphi (\bx, \bbeta) = (\frac{\beta_{\byi \to \byj}}{\bx^{\byi}})_{\byi \to \byj \in E} \subseteq \mathbb{R}_{>0}^{E}.
\end{equation}
Since $\bbeta$ is a complex-balanced flux vector, we get
\begin{equation} \notag
\sum_{\byi\rightarrow \byj} \varphi_{\byi \to \byj} \bx^\byi
= \sum_{\byj\rightarrow \byi} \varphi_{\byi \to \byj}  \bx^\byj.
\end{equation} 
Hence, $\varphi (\bx, \bbeta)$ is a complex-balanced rate vector with the complex-balanced steady state $\bx \in \mathcal{S}_{{\boldsymbol{x}_0}}$ on $G$. 
Therefore, we conclude that $\varphi (\bx, \bbeta) \in \mathcal{V}(G)$, and $\varphi$ is well-defined.
Further, from Definition \ref{def: varphi} we can directly get that $\varphi$ is a continuous map.
\end{proof}

\begin{lemma} \label{lem:bijective}
For any state ${\boldsymbol{x}_0} \in \mathbb{R}_{>0}^n$, the map $\varphi$ is bijective.
\end{lemma}

\begin{proof}
First, we show $\varphi$ is surjective.
By Theorem \ref{thm:HJ}, for any reaction rate vector $\bk \in \mathcal{V}(G)$, there exists a (unique) complex-balanced steady state $\bx \in \mathcal{S}_{{\boldsymbol{x}_0}}$. Then we define a flux vector $\bbeta = (\beta_{\byi \to \byj})_{\byi \to \byj \in E}$ as follows:
\begin{equation} \notag
\beta_{\byi \to \byj} := k_{\byi \to \byj} \bx^{\byi}.
\end{equation} 
Using Lemma \ref{lem:cbf connection with graph}, we derive that $\bbeta \in \bB (G)$, and 
$\varphi(\bx,\boldsymbol\beta)={\bk}$.

\smallskip

Next, we show $\varphi$ is injective.
Assume that $(\hat{\bx}, \hat{\bbeta}), (\tilde{\bx}, \tilde{\bbeta}) \in \mathcal{S}_{{\boldsymbol{x}_0}} \times \mathcal{B}(G)$, such that 
\[
\varphi (\hat{\bx}, \hat{\bbeta}) = \varphi (\tilde{\bx}, \tilde{\bbeta}).
\]
Following \eqref{eq:phixb}, we derive two reaction rate vectors $\hat{\varphi}$ and $\tilde{\varphi}$ as follows:
\begin{equation} \label{eq: injective}
\varphi (\hat{\bx}, \hat{\bbeta}) :=
\bigg( \frac{\hat\beta_{\byi \to \byj}}{\hat{\bx}^{\byi}} \bigg)_{\byi \to \byj \in E}, \ \text{and } \
\varphi (\tilde{\bx}, \tilde{\bbeta}) :=
\bigg( \frac{\tilde\beta_{\byi \to \byj}}{\tilde{\bx}^{\byi}} \bigg)_{\byi \to \byj \in E}
\end{equation}
From $\varphi (\hat{\bx}, \hat{\bbeta}) = \varphi (\tilde{\bx}, \tilde{\bbeta})$ and Lemma \ref{lem:phiContin}, the uniqueness on the complex-balanced steady state within each affine invariant polyhedron, we obtain $\hat{\bx} = \tilde{\bx}$. Then from Equation \eqref{eq: injective}, it is clear that $\hat{\bbeta} = \tilde{\bbeta}$, and we conclude the injectivity.
\end{proof}

\begin{lemma} \label{prop:contInv}
For any state ${\boldsymbol{x}_0} \in \mathbb{R}_{>0}^n$, the map $\varphi^{-1}$ is well-defined, and continuous.
\end{lemma}

\begin{proof}
Since we have proved that the map $\varphi$ is bijective in Lemma \ref{lem:bijective}, it is standard that $\varphi^{-1}$ is well-defined.

\smallskip

Now we show that $\varphi^{-1}$ is continuous.
From Lemma \ref{lem:phiContin}, given any $(\bx, \bbeta) \in \mathcal{S}_{{\boldsymbol{x}_0}} \times \mathcal{B}(G)$, $\varphi (\bx, \bbeta)$ forms a complex-balanced rate vector with $\bx$ being the complex-balanced steady state. 
Since $\varphi$ is bijective and the complex-balanced steady state is unique in $\mathcal{S}_{{\boldsymbol{x}_0}}$, for any complex-balanced rate vector $\bk = (k_{\byi \to \byj})_{\byi \to \byj \in E} \in \mathcal{V}(G)$, we have
\begin{equation}
\varphi^{-1} (\bk) = (\bx, \bbeta),
\end{equation}
such that
\begin{equation} \label{x beta in inverse}
\bx = Q_{{\boldsymbol{x}_0}}(\bk) \ \ \text{and } \
\beta_{\byi \to \byj} := k_{\byi \to \byj} \bx^{\byi}, \ \text{with } \
\bbeta = (\beta_{\byi \to \byj})_{\byi \to \byj \in E}.
\end{equation}
Applying Theorem \ref{th:qCont}, we get the map $Q_{{\boldsymbol{x}_0}}$ is continuous, which says that 
$\bx$ depends continuously on $\bk$. 
Moreover, every component in $\bbeta$ can be written as a polynomial of $\bk$ and $\bx$. This reveals that 
$\bbeta$ also depends continuously on $\bk$. 

After showing that both components in the product space $\mathcal{S}_{{\boldsymbol{x}_0}} \times \mathcal{B}(G)$ vary continuously on $\bk$, we conclude the continuity on the map $\varphi^{-1}$.
\end{proof}

Finally, we are able to prove Theorem \ref{th:product}.

\begin{proof}[\textbf{Proof of Theorem \ref{th:product}}]
From Definition \ref{def:homeo}, it suffices to show that the map $\varphi$ is a homeomorphism. 
Applying Lemma \ref{lem:bijective}, we derive $\varphi$ as a bijective function. 
From Lemma \ref{lem:phiContin} and Lemma \ref{prop:contInv}, we show that both $\varphi$ and $\varphi^{-1}$ are continuous functions.
Therefore, we conclude $\varphi$ is a homeomorphism, and prove this theorem.
\end{proof}

\subsection{Connection to deficiency theory}
\label{sect:deficiency}

The notion of deficiency of a reaction network or E-graph was introduced by  Feinberg and Horn~\cite{MR413930, horn1972necessary}. It is an invariant of the network and plays a key role in the study of complex-balanced steady states of a network~\cite{feinberg, HJ}. 

\begin{definition}[\cite{feinberg, cy}]
\label{def:defic}
Consider an E-graph $G = (V, E)$ with $\ell$ connected components and $m$ vertices. 
Let $s$ be the dimension of the stoichiometric subspace $\mathcal{S}$. The \defi{deficiency} of an E-graph $G$ is the non-negative integer 
\begin{equation}
\delta := m - l -s.
\end{equation}
\end{definition}

Under mass-action kinetics, networks with low deficiency have special dynamical properties. For example, the deficiency zero theorem
shows that weakly reversible deficiency zero networks are complex-balanced for any choices of rate constants ~\cite{MR413930, horn1972necessary}.
In \cite{CDSS}, it was shown that given a weakly reversible E-graph $G$, the set $\mathcal{V}(G)$ is an algebraic variety of codimension $\delta$ in $\mathbb{R}_{>0}^{E }$.
In the following, we will recover this result by using the product structure of the moduli space $\mathcal{V}(G)$ from Theorem \ref{th:main}.

\begin{proposition} \label{prop:dim}
Consider an E-graph $G = (V, E)$ with $\ell$ connected components and $m$ vertices. 
Let $s$ be the dimension of the stoichiometric subspace $\mathcal{S}$, then
\begin{equation*}
\mathrm{dim}(\mathcal{V}(G)) = \vert E\vert - m + s + l.
\end{equation*}
\end{proposition}

\begin{proof}
Recall that the dimension of a product of topological spaces is a topological invariant and it is given by the sum of the dimensions of the factors \cite{lee}.
In addition, the dimension of a variety at a regular point is the dimension of its tangent vector space at that point, thus it is the same dimension as seen as a manifold as well as seen as a variety \cite{kendig}.

Now using Theorem \ref{th:product}, we have for any state ${\boldsymbol{x}_0} \in \mathbb{R}_{>0}^n$, 
\begin{equation} \label{dim VG}
\mathrm{dim}(\mathcal{V}(G))=\mathrm{dim}\big (({\boldsymbol{x}_0}+\mathcal{S})\cap \mathbb{R}_{>0}^n\big )+\mathrm{dim}(\mathcal{B}(G)),
\end{equation}
and it is clear that
$\mathrm{dim}\big (({\boldsymbol{x}_0}+\mathcal{S})\cap \mathbb{R}_{>0}^n\big ) = \dim (\mathcal{S}) = s$.

Recall that $\mathcal{B}(G) \subseteq \mathbb{R}_{>0}^{E}$ represents the set of complex-balanced flux vectors that satisfy \eqref{eq:beta}. 
Following Kirchhoff junction rules, for each connected component of $G$ with $m_i$ vertices, there are $m_i - 1$ independent conditions among the linear conditions defining $\mathcal{B}(G)$ in \eqref{eq:beta}.
Further, we can check that linear conditions are independent when working on different connected components of $G$. Hence, we get
\begin{equation*}
    \mathrm{dim}(\mathcal{B}(G))
    = \vert E\vert - \sum\limits^{l}_{i=1} (m_i -1 )
    =\vert E\vert -m+l.
\end{equation*}
Together with \eqref{dim VG}, we conclude the proposition.
\end{proof}

The following corollary is a direct consequence of Proposition \ref{prop:dim}. It was first proved by a different method in \cite{CDSS}.

\begin{corollary} \label{cor:defic}
Let $G = (V, E)$ be a weakly reversible E-graph. Then the codimension of the moduli space $\mathcal{V}(G) \subseteq \mathbb{R}_{>0}^{E}$ is  $\delta$.
\end{corollary}

\begin{proof}
The codimension on $\mathcal{V}(G)$ follows
\begin{equation*}
\mathrm{codim}(\mathcal{V}(G))
=\vert E \vert - \dim (\mathcal{V}(G))
=\vert E \vert -(\vert E\vert -m+s+l)=\delta.
\end{equation*}
\end{proof}

\subsection{Bijective affine transformations preserve the toric locus} \label{sect:bijectiveAffineTransf}

In this subsection, we prove that the toric locus is preserved by bijective affine transformations of the network.

\begin{definition}
Consider a network $G = (V, E)$ in $\mathbb{R}^n$. 
Suppose $T: \mathbb{R}^n \to \mathbb{R}^n$ is a bijective affine transformation. Denote by 
\[
T(V) := \{ T(\boldsymbol{y}) \mid \boldsymbol{y} \in V\}, \text{ and } \
T(E) := \{ T(\boldsymbol{y}_i) \to T(\boldsymbol{y}_j) \mid  \by_i \to \by_j \in E\}.
\]
Then we call the graph $T(G) := (T(V), T(E))$ the \defi{bijective affine image of $G$ by $T$}. 
\end{definition}

\begin{theorem} \label{thm:bijectiveAffine}
Consider a weakly reversible E-graph $G_1$. If $G_2$ is a bijective affine image of the graph $G_1$, then $G_1$ and $G_2$ have the same toric locus. Namely, 
\[
\mathcal{V}(G_1)=\mathcal{V}(G_2).
\]
\end{theorem}

\begin{proof}
The result follows from \cite[Theorem 9]{CDSS} and from the Matrix-Tree theorem. In particular see \cite[Section 2, page 5]{CDSS}: the change of coordinates given by the spanning trees in the two graphs are the same.
\end{proof}

\section{Discussion and future work}
\label{sec:dis}

There has been strong interest in the study of the moduli space of toric dynamical systems, i.e., the set of parameters that give rise to complex-balanced dynamical systems. This interest is due to the very stable dynamical behaviour of these systems; for instance, the complex-balanced steady states are known to be \emph{locally asymptotically stable} within their invariant polyhedron.

Since important properties of complex-balanced dynamical systems can be analyzed using Nonlinear Algebra tools (see for example \cite[Chapter 6]{breiding2021nonlinear}), the authors of \cite{CDSS}  called these systems \emph{toric dynamical systems} (see also \cite[Chapter 5]{MR4284895}). Indeed, not only the moduli spaces of toric dynamical systems are toric, but also  the steady-state locus (i.e. the fixed points) of toric dynamical systems can be described by binomial equations; see \cite{Ga01}. Another computational advantage of this fact is that one may  describe the steady states of such a system in terms of monomial parametrizations (\cite{MR4152911}). The fruitful combinatorial and computational properties of binomial ideals are well-known and they are desirable in  applications, since toric algebraic varieties are well-understood.

In this paper we prove that, given a complex-balanced mass-action system and positive initial data, the positive complex-balanced equilibria vary continuously in function of the parameters.
Next, using this result, we show that the toric locus is connected and we emphasize the product structure of the toric locus. Namely, we prove that there exists a homeomorphism between the toric locus and the product of the set of complex-balanced flux vectors and the affine invariant polyhedron.
We provide an explicit parametrization in terms of the parameters (i.e., the reaction rate constants), as shown in \eqref{eq:phixb}.

In future work \cite{smooth2022}, we will use some of the approaches developed here to show that the positive complex-balanced equilibria of a complex-balanced mass-action system actually depend smoothly on the reaction rate constants and on the initial data.
Furthermore, the approach used on showing the homeomorphism will allow us to derive regularity properties of the toric variety $\mathcal{V}(G)$.

\section*{Acknowledgements}
The authors gratefully acknowledge the support of Bernd Sturmfels and of the Max Planck Institute for Mathematics in the Sciences in Leipzig, Germany.  G.~Craciun was  partially supported by the National Science Foundation  grant  DMS--2051568. M.-{\c S}. Sorea is also grateful to Antonio Lerario for the very supportive working conditions during her postdoc at SISSA, in Trieste, Italy.

\printbibliography

\noindent
\footnotesize 

{\bf \noindent Authors:}

\noindent Gheorghe Craciun\\
University of Wisconsin-Madison, USA\\
 {\tt craciun@math.wisc.edu}
\vspace{0.5cm}

\noindent Jiaxin Jin\\ Ohio State University, USA\\
{\tt jin.1307@osu.edu}
\vspace{0.5cm}

\noindent Miruna-\c Stefana Sorea\\ SISSA (Scuola Internazionale Superiore di Studi Avanzati), Trieste, Italy and Lucian Blaga University, Sibiu, Romania\\
{\tt msorea@sissa.it, mirunastefana.sorea@ulbsibiu.ro}

\end{document}